\documentclass{amsart}
\usepackage{amscd,amsthm,mathtools}
\usepackage{latexsym}
\usepackage{comment}

\DeclareFontFamily{U}{wncy}{}
\DeclareFontShape{U}{wncy}{m}{n}{<->wncyr10}{}
\DeclareSymbolFont{mcy}{U}{wncy}{m}{n}
\DeclareMathSymbol{\Sha}{\mathord}{mcy}{"58}

\usepackage{amssymb,amsmath}
\usepackage{amsfonts}
\usepackage{tkz-graph}
\usepackage{hyperref}

\newcommand{\Z}{{\mathbb Z}}
\newcommand{\isi}{{\mathcal I}}
\newcommand{\Q}{{\mathbb Q}}
\newcommand{\R}{{\mathbb R}}

\newcommand{\C}{\mathcal{C}}
\newcommand{\HH}{\mathcal{H}}
\newcommand{\K}{\mathcal{K}}
\newcommand{\Kum}{\mathcal{K}}

\newcommand{\OO}{\mathcal{O}}
\newcommand{\E}{\mathcal{E}}
\newcommand{\A}{\mathbb{A}}
\newcommand{\F}{\mathcal{F}}
\renewcommand{\P}{\mathbb{P}}

\newcommand{\dO}{{\mathfrak{d}}_\OO}
\newcommand{\p}{{\mathfrak{p}}}
\DeclareMathOperator{\Jac}{Jac}
\DeclareMathOperator{\Hom}{Hom}
\DeclareMathOperator{\dia}{dia}
\DeclareMathOperator{\hh}{h}
\DeclareMathOperator{\Proj}{Proj}
\newcommand{\JacC}{{\hbox{Jac}_{\lower.5pt\hbox{$_\C$}}}}
\newcommand{\JacF}{{\hbox{Jac}_{\lower.5pt\hbox{$_\F$}}}}
\newcommand{\etale}{{\mathrm{\acute{e}t}}}

\DeclareMathOperator{\HHH}{H}

\DeclareMathOperator{\Gal}{Gal}

\newcommand{\elllmfdb}[3]{\href{https://www.lmfdb.org/EllipticCurve/#1/#2/#3}{#1/#2-#3}}

\theoremstyle{plain}
\newtheorem{thm}{Theorem}[section]
\newtheorem{conj}[thm]{Conjecture}
\newtheorem{claim}[thm]{Claim}
\newtheorem{lemma}[thm]{Lemma}
\newtheorem{alg}[thm]{Algorithm}
\newtheorem{prop}[thm]{Proposition}
\newtheorem{cor}[thm]{Corollary}
\theoremstyle{definition}

\newtheorem{notation}[thm]{Notation}
\newtheorem{example}[thm]{Example}
\newtheorem{remark}[thm]{Remark}
\newtheorem{defn}[thm]{Definition}

\theoremstyle{remark}

\def\R{{\mathbb R}}
\def\Q{{\mathbb Q}}
\def\N{{\mathbb N}}
\def\F{{\mathbb F}}
\def\Z{{\mathbb Z}}
\def\C{{\mathbb C}}
\def\O{{\mathcal O}}

\DeclareMathOperator{\Pic}{Pic}

\DeclareMathOperator{\GL}{GL}
\DeclareMathOperator{\SL}{SL}

\begin{document}
\bibliographystyle{plain}
\bibstyle{plain}

\title[Computations on Hilbert modular surfaces]{Rings of Hilbert modular forms, computations on Hilbert modular surfaces, and the Oda-Hamahata conjecture}

\author{Adam Logan}
\address{The Tutte Institute for Mathematics and Computation,
P.O. Box 9703, Terminal, Ottawa, ON K1G 3Z4, Canada}
\address{School of Mathematics and Statistics, 4302 Herzberg Laboratories, 1125 Colonel By Drive, Carleton University, Ottawa, ON K1S 5B6, Canada}
\address{Department of Pure Mathematics, University of Waterloo, 200 University Ave.~W., Waterloo, ON N2L 3G1, Canada}
\email{adam.m.logan@gmail.com}

\date{\today}

\maketitle

\begin{abstract}
  The modularity of an elliptic curve $E/\Q$ can be expressed either as an
  analytic statement that the $L$-function is the Mellin transform of
  a modular form, or as a geometric statement that $E$ is a quotient of a
  modular curve $X_0(N)$.  For elliptic curves over number fields
  these notions diverge; a conjecture of Hamahata
  asserts that for every elliptic
  curve $E$ over a totally real number field there is a correspondence between
  a Hilbert modular variety and the product of the conjugates of $E$.
  In this paper we prove the conjecture by explicit computation
  for many cases where $E$ is defined over a real quadratic field and the
  geometric genus of the Hilbert modular variety is $1$.
\end{abstract}

\section{Introduction}
One of the most fundamental discoveries in the arithmetic of elliptic curves
over $\Q$ is the modularity theorem, previously the Shimura-Taniyama-Weil
conjecture and later proved in \cite{BCDT}, building on work of Wiles,
Taylor-Wiles, and many other mathematicians.  We may state it briefly in
two ways:

\begin{thm}\label{thm:modularity}
  Let $E$ be an elliptic curve over $\Q$ and let $L(s,E)$ be its
  $L$-function as defined in \cite[Ex.~8.19]{silverman}.
  Let $N = N_E$ be the conductor
  of $E$ and let $X_0(N)$ be the modular curve parametrizing elliptic curves
  with a cyclic subgroup of order $N$.  Then:
  \begin{enumerate}
  \item[(analytic)] $L(s,E)$ is the Mellin transform of a Hecke eigenform
    of weight~$2$ and level $N$.  (Concretely, the coefficients of the
    Dirichlet series of $L$ are the same as those of the Fourier expansion
    of $f$.)
  \item[(geometric)] There is a dominant map $X_0(N) \to E$.
  \end{enumerate}
\end{thm}

Our main goal in this paper is to prove a few special cases of a form of
geometric modularity for elliptic curves over real quadratic fields. 
We will state the result and then devote the
rest of the introduction to explaining the statement.

\begin{thm}\label{thm:oda-special}
  Let $E$ be an elliptic curve over a real quadratic field belonging to
  an isogeny class with one of the following LMFDB labels or their
  Galois conjugates:
  \elllmfdb{2.2.5.1}{31.1}{a},
  \elllmfdb{2.2.5.1}{55.1}{a},
  \elllmfdb{2.2.5.1}{80.1}{a},
  \elllmfdb{2.2.8.1}{9.1}{a},
  \elllmfdb{2.2.8.1}{17.1}{a},
  \elllmfdb{2.2.8.1}{32.1}{a},
  \elllmfdb{2.2.13.1}{9.1}{a},
  \elllmfdb{2.2.17.1}{4.1}{a},
  \elllmfdb{2.2.17.1}{8.3}{a}.
  Then there is a correspondence (not necessarily defined over $\Q$)
  between the Hilbert modular variety of
  level equal to the conductor of $E$ and the Kummer surface of
  $E \times E^\sigma$, inducing a map on $\HHH^2$ that is injective on the
  transcendental lattice of the Kummer surface.
\end{thm}

To give some context for this result,
let us more generally consider an elliptic curve $E$ over
a totally real number field $K$.
In order to describe the situation we begin with a brief explanation of the
concept of a Hilbert modular form.

\begin{defn}\label{def:hmf} \cite[I.1, II.7]{vdg}
  Let $K$ be a totally real number field of degree $d$ over $\Q$,
  let $\HH$ be the upper half-plane, and let $\HH_{\pm}$ be the union of the
  upper and lower half-planes.
  Then $\GL_2^+(\R)$ acts on $\HH$ and so $\GL_2^+(\R)^d$ acts on
  $\HH^d$.  Choosing an ordering
  $\rho_1, \dots, \rho_d$ of the real embeddings of $K$, we obtain an
  embedding $\GL_2^+(K) \hookrightarrow \GL_2^+(\R)^d$ and hence an action of
  $\GL_2^+(K)$ on $\HH^d$.
\end{defn}

\begin{defn} Let $I \subseteq \O_K$ be a nonzero ideal.  Then
  $\Gamma_0(I)$ is the subgroup of $\SL_2(\O_K)$ of matrices whose lower
  left entry belongs to $I$.  (One also adapts the usual definitions for
  $\SL_2(\Z)$ to define $\Gamma_1(I), \Gamma(I)$, but these will not be used
  in this paper.)
\end{defn}

We remark \cite[p.~6, p.~15]{vdg}
that $\SL_2(\O_K)$ acts properly discontinuously and
with finite stabilizers on $\HH^d$.  This is the reason to act on
$\HH^d$; if we only used some of the embeddings to define the action,
then every orbit would be dense and the quotient would not have a reasonable
topology.

\begin{defn}\cite[Definition I.6.1]{vdg}
  Let $d>1$ and let $w_1,\dots,w_d$ be integers.
  A {\em Hilbert modular form}
  of weights $(w_1, \dots, w_d)$ and level $I$ is a holomorphic function
  $f$ on $\HH^d$ such that, for every
  $M = \begin{psmallmatrix}p&q\\r&s\end{psmallmatrix} \in \Gamma_0(I)$
  and $(z_1,\dots,z_d) \in \HH^d$,
  we have
  $$f(z_1,\dots,z_d) = \prod_{i=1}^d (\rho_i(r) z_i + \rho_i(s))^{w_i}
  f\left(\frac{\rho_1(p)z_1 + \rho_1(q)}{\rho_1(r)z_1 + \rho_1(s)}\right),\dots,
  f\left(\frac{\rho_d(p)z_1 + \rho_d(q)}{\rho_d(r)z_1 + \rho_d(s)}\right).$$
\end{defn}

We remark that Hilbert modular forms constitute a $\Z^d$-graded ring, the
product of forms of weights $(w_1,\dots,w_d)$ and $(w_1',\dots,w_d')$ being
a form of weight $(w_1+w_1',\dots,w_d+w_d')$.
As with classical modular forms, Hilbert modular forms have a sort of Fourier
expansion with coefficients indexed by totally positive elements of the
inverse different of $\O_K$.  For certain purposes it is better to
consider the action of a slightly larger group, such as a subgroup of
$\GL_2(\O_K)$, on the product of $d$ copies of $\HH_\pm$, the union of the
upper and lower half-planes.  For this paper it is mostly not necessary to
do this and so we leave the minor modifications necessary to the reader.

We now return to elliptic curves over $K$.  Let $E$ be such a curve and let
its conductor be $I$.  To generalize the algebraic statement of modularity
is fairly straightforward.  One requires that the coefficients of the
$L$-function of $E$ match the Hecke eigenvalues of a Hilbert modular form $f$.
This is known for all elliptic curves over real quadratic fields and many
over fields of higher degree: see \cite[Theorems 1-2]{FLS}.

To state a reasonable geometric modularity conjecture requires some thought.
Before doing so, we introduce the Hilbert modular variety associated to
a group $\Gamma$.  For simplicity we will only do this over $\C$ by a
complex-analytic construction.

\begin{defn}\label{def:hmv}
  Let $\Gamma$ be a subgroup of $\GL_2^+(K)$ commensurable with
  $\SL_2(\O_K)$.  Then the {\em Hilbert modular variety}
  $\Gamma \backslash \HH^d$ is the quotient of $\HH^d$ by $\Gamma$ in the
  category of complex analytic spaces.  If $\Gamma = \Gamma_0(I)$, then
  we write $Y_0(I)$ for the quotient.
\end{defn}

From this point of view it is not at all obvious that $\Gamma \backslash \HH^d$
has an algebraic structure, much less that it can be defined over a number
field, but this is true \cite[Thm.~IV.7.1]{vdg}.

\begin{prop}\label{prop:x0}
  Let $J_1, \dots, J_{h_I}$ be ideals representing the cosets of the
  group of fractional ideals of $\O_K$ modulo those generated by an
  element that is $1 \bmod I$.  Then 
  there exists a projective variety that is constructed from $Y_0(I)$ by
  adding points $P_1, \dots, P_{h_I}$ in canonical bijection with the $J_i$.
\end{prop}

This variety, which we will denote $H_{I,K}$,
is the famous {\em Baily-Borel compactification} \cite[II.7]{vdg}, which is to
say that it is Proj of the ring of Hilbert modular forms defined above.

\begin{remark}\label{rem:even-weight}
  In contrast to the classical case where the presence of $-I_2$ in
  $\Gamma_0(N)$ implies that there are no nonzero cusp forms of odd weight,
  such forms may exist here.  However, we will always consider
  Proj of the subring of Hilbert modular forms of even weight.  This is
  the second Veronese embedding and is therefore isomorphic to Proj of the
  full ring by \cite[Exercise 7.4.D]{risingsea}.
\end{remark}

What is the geometric relation between $H_{I,K}$ and an elliptic curve $E$
over $K$ of conductor $I$?  It is quite unreasonable to expect a nonconstant
map $H_{I,K} \to E$ as in the classical case, since $H_{I,K}$ is defined over
$\Q$ and $E$ only over a number field.  We might replace $E$ by the product
of its conjugates, which can be defined over $\Q$, but this will still not
admit any nonconstant maps from $H_{I,K}$.  The reason is that
$h^1(H_{I,K}) = 0$ for $d > 1$ \cite[IV.6, p.~82]{vdg},
and so $H_{I,K}$ does not admit
any nonconstant maps to an abelian variety.

The version of the modularity conjecture to be discussed here
can be expressed in two forms.
The weak form, which could still be considered algebraic, is that
the Galois representation on $\HHH^d_{\etale}(H_{I,K})$ has a subquotient
isomorphic to $\otimes_{i=1}^d \HHH^1_{\etale} \rho_i(E)$, where as above the
$\rho_i$ are the embeddings of $K$ into $\R$.  The strong form was first
stated by Hamahata:

\begin{conj}\cite[Section 0, p.~194]{hamahata}\label{conj:oda-hamahata}
  Let $E$ be an elliptic curve over a real quadratic
  field $K$ with conductor $I$.  Then there is a correspondence between
  $H_{I,K}$ and $K_E = (E \times E^\sigma)/\pm 1$ that induces an injection from
  the primitive part of $\HHH^2(K_E)$ to that of $\HHH^2(H_{I,K})$.
\end{conj}

\begin{remark} As Hamahata notes, this conjecture is inspired by the work
  of Oda \cite{oda}; we therefore refer to it as the {\em Oda-Hamahata
    conjecture}.  It is an important special case of the Tate conjecture.
  An especially strong form of this conjecture would state that
  this correspondence can be defined over $\Q$, not just over $\bar \Q$.
\end{remark}

\begin{remark} Hamahata does not define the term ``primitive'', but we
  understand the phrase ``primitive part of $\HHH^2$'' to refer to the
  transcendental lattice, namely the quotient of $\HHH^2$ by the
  subspace spanned by fundamental classes of curves.  We also remind the
  reader that a correspondence $X \subset V \times W$ induces maps
  on cohomology \cite[p.~3]{vanhoften}.  We have
  $\HHH^*(V \times W) \cong \HHH^*(V) \otimes \HHH^*(W)$ by the K\"unneth
  formula, and $\HHH^*(V) \cong \HHH^*(V)^\vee$ by Poincar\'e duality.
  Thus $\HHH^*(V \times W) \cong \Hom(\HHH^*(V),\HHH^*(W))$, and one sees
  easily that if $X \subset V \times W$ with $\dim X = \dim V = \dim W$
  (the only case used here) the map preserves degrees.  Our correspondences
  will not be explicit subvarieties of $V \times W$; rather, they are induced
  from pairs of rational maps
  $\phi: V \dashrightarrow T, \psi: W \dashrightarrow T$
  as the closure of the locus of $(v,w)$ where $\phi(v), \psi(w)$ are both
  defined and are equal.
\end{remark}

\begin{remark} More generally, let $K$ be a totally real number field of
  arbitrary degree and let $E$ be an elliptic curve over $K$ of conductor
  $I$.  Let $A = \prod_{i=1}^{[K:\Q]} E^{\sigma_i}$, where the $\sigma_i$ are the
  embeddings of $K$ into $\R$.
  Then $A$ admits
  an action of the group $G = (\Z/2\Z)^d$, where each generator acts as
  negation on one factor and the identity on the others.
  Let $G_e$ be the subgroup of elements whose sum is $0$ (in particular, if
  $d = 2$ we have $G_e = \{\pm 1\}$ and the quotient is the
  classical Kummer surface \cite{proleg} as above).
  The quotient $A/G_e$ is a Calabi-Yau variety and the primitive part of
  $\HHH^d(A/G_e)$ can be identified with $\otimes_i \HHH^1(E^{\sigma_i})$..
  It is natural to generalize the Oda-Hamahata conjecture to this situation:
  namely, we ask for a correspondence between $A/G_e$ and $H_{I,K}$ inducing an
  injection from the primitive part of $\HHH^d(A/G_e)$ to that of
  $\HHH^d(H_{I,K})$.  For now this conjecture appears to be beyond computational
  attack for $d > 2$.
\end{remark}

Hamahata's conjecture is known under some special circumstances:
in particular, if $E$ is a {\em $\Q$-curve}, i.e., an elliptic curve
isogenous to all of its Galois conjugates.  This is not clearly stated or
proved in the literature, but is known to experts.  It is a consequence
of the modularity of $E$, which is proved in \cite[Corollary~6.2]{ribet},
subject to Serre's conjecture on modular $2 \times 2$ Galois representations,
which is now a theorem of Khare and Wintenberger \cite{kw}.
More precisely, Ribet's result is that $E$ is a quotient of $J_1(N)$; it
follows that there is a dominant map $J_1(N)^d \to \prod_{i=1}^{d} E^{\sigma_i}$,
where $d = [K:\Q]$.  The domain of this map is related to $\HHH^d(X_0(N))$
by Hecke correspondences, and the equality of the levels follows from
Serre's formula for the level of the modular form, also proved by Khare
and Wintenberger.

Our main goal in this paper is to prove some new cases of this conjecture.
This advance is possible because
of the recent release of software \cite{hmf} for computing rings of
Hilbert modular forms.

\begin{remark} Let $\O$ be an order in a totally real field and $\O^+$
  the set of its totally nonnegative elements.  Let $\dO$ be the different of
  $\O$.  A Hilbert modular form for $\O$ has a Fourier
  expansion in which the coefficients are indexed by elements of 
  $\O^+ \cap \dO^{-1}$ and are invariant under multiplication by
  elements of $\O^+ \cap \O^*$.  In order to multiply two such expansions,
  one needs an algorithm for writing an element of $\O^+$ as a sum of two
  such elements in all possible ways.  An efficient method for doing so
  will be described in \cite{hmf-mult}.
\end{remark}

We state our main result here:
\begin{thm} Let $K$ be a real quadratic field of narrow class number $1$
  and $I$ an ideal of $\O_K$.  Let $H_{I,K}$ be the associated Hilbert
  modular surface; assume that $I$ is an elliptic curve on the list
  in Theorem~\ref{thm:oda-special}.
  Then there is an elliptic curve $E/K$ with conductor $I$
  and an explicit correspondence between $H_{I,K}$ and $E \times E^\sigma$
  (possibly defined only over a nontrivial extension of $\Q$)
  that identifies the transcendental parts of $H^2$ of the two surfaces.
\end{thm}

\begin{remark} In particular, this theorem applies in all cases where
  $H_{I,K}$ is not of general type and $I \ne (1)$,
  except for $K = \Q(\sqrt{13}), I = (2)$.
  In this case we do find a K3 surface that is in correspondence with the
  Hilbert modular variety (and in fact this has been known for some years,
  being \cite[Theorem 9]{vdgz}) but it is quite difficult to establish
  an explicit correspondence with the desired Kummer surface.
  In two of the cases where $I = (1)$ the birational equivalence between the
  Hilbert modular surface and a K3 surface with canonical singularities is
  made explicit in \cite[Appendix]{williams}.
\end{remark}

We have also considered a few cases in which the field $K$ has
narrow class number $2$.
One of these will be discussed in Section~\ref{ex:d12-n13}.

Our method is to study the Baily-Borel compactification
$H_{I,K}$ by means of computer algebra and exhibit the desired correspondence
to the Kummer surface by explicit maps.  In some cases the Kummer surface
of $E \times E^\sigma$ is itself a quotient of $H_{I,K}$; in other cases
this does not appear to be true and instead we find an auxiliary K3 surface
$S$ that admits maps of finite degree from both $H_{I,K}$ and the Kummer
surface of $E \times E^\sigma$.  The methods are similar to those of
\cite[Appendix]{williams}, except that we do not necessarily start with a
variety birational to a K3 surface.  In addition, the setup of
\cite{williams} requires information about the ring of Hilbert
modular forms that is not available here: in particular, that it can be
generated by Borcherds products.  Against this, we do not find equations for
Hirzebruch-Zagier cycles \cite[Def.~VI.1.4]{vdg}
on the Hilbert modular surfaces.

The Hilbert modular surfaces are not only interesting for their relation
to elliptic curves over real quadratic fields; they are worth studying
in their own right.  In particular, some such surfaces, or their quotients
by Atkin-Lehner involutions, furnish examples of surfaces of geometric
genus $0$ and small $h^{1,1}$.  It is interesting to relate these
to known examples of surfaces with these properties, such as the famous
Godeaux and Campedelli surfaces (see \cite{reid-godeaux}, for example).
As suggested by Hamahata \cite[p.~194]{hamahata}, we would expect that
the Hilbert modular threefolds (and their Atkin-Lehner quotients) also give
examples of threefolds with unusual properties.  For example, in light of
the work of Grundman \cite{Grundman1992, Grundman1994}, some such threefolds
at levels where the geometric genus is $0$ are varieties of general type
whose Hodge diamond is that of a rational variety.  Hilbert
modular threefolds with geometric genus $1$ may also be interesting examples
of varieties of general type with small Hodge numbers, although the
classification problem, already very difficult for surfaces, is presumably
hopeless in this situation.

In other cases, the surface or a quotient has geometric genus $1$ but is
of general type.  It is known \cite[Theorem 1]{morrison}
that if a surface $S$ is of
geometric genus $1$ then there is a K3 surface $K_S$ and a correspondence
between $S$ and $K_S$ identifying the transcendental part of $H^2_{\etale}$
for the two surfaces.  However, finding the K3 surface and correspondence
explicitly is a difficult and interesting problem in general.
In some such cases the quotient by some or all of
the Atkin-Lehner involutions is a surface with $p_g = 1$ and $\kappa < 2$
and we can find the desired correspondence, but otherwise there is no
obvious method to use.  This problem will be studied in future work.

\subsection*{Acknowledgments} This paper began when Wei Zhang asked John
Voight whether advances in theory and computation in recent decades would
make it possible to prove some new cases of Oda's conjecture.
Prof.~Voight passed the question along to me; I thank him for doing so and
for many interesting and informative discussions along the way, including
some essential references.  In addition,
this paper could not have been written without the software \cite{hmf}.
I am grateful to its developers and especially to Eran Assaf and Edgar Costa
for their explanations of how best to use the code.  My
understanding of modularity, and in particular of the distinction between
the algebraic and geometric viewpoints, owes much to discussions
with Jared Weinstein.  I also thank Dami\'an Gvirtz-Chen for suggesting
an interesting calculation.

While writing this paper I benefited from the hospitality of ICERM and the
support of the Simons Collaboration on Algebraic Geometry, Number Theory,
and Computation (grant $546235$), as well as the liberality of the
Tutte Institute for Mathematics and Computation in permitting me to take
leave from my employment there and the opportunity to visit the Department of
Pure Mathematics at the University of Waterloo.


\section{General approach to constructing a K3 surface}
We start by establishing the notation.
\begin{notation}\label{not:basic}
  Let $K$ be a real quadratic field, let $\O_K$ be its maximal order, and let
  $I$ be an ideal of $\O_K$.  We define $M_I = M_{I,K}$ to be the ring of
  Hilbert modular forms of level $I$ and parallel weight.  Further, we let
  $H_I = H_{I,K} = \Proj M_{I,K}$ be the well-known
  {\em Baily-Borel compactification} \cite[II.7]{vdg}
  of the Hilbert modular surface of level $I$.
\end{notation}

In this paper we only use forms of parallel weight,
so we will abuse notation by writing $w_i$ instead of $(w_i,w_i)$ for the
weight of a Hilbert modular form.  In addition, the presence of $-I_2$ in
$\Gamma_0(N)$ implies that all forms are of even weight.

\begin{defn} Let $f_1, \dots, f_n$ be a set of modular forms that generate
  $M_I$.  Let their weights be $w_1, \dots, w_n$.  Then we define $\P_I = \P_{I,K}$ to be
  the projective space over $\Q$ with variable degrees $w_i/2$.
\end{defn}
As above, $H_{I,K}$ is naturally a subvariety of $\P_I$.
The ring $M_I$ is computed by the software package \cite{hmf}.

\subsection{Atkin-Lehner involutions}
The ring $M_I$ and the surface $H_I$ are rather complicated.  When the
level is not $1$, they always admit Atkin-Lehner involutions, and the
quotient of $H_I$ by these, corresponding to the fixed subring of $M_I$,
is simpler.  Recall \cite{greenberg-voight}
that if $I = I_1I_2$ with $I_1 + I_2 = (1)$, then there is an
{\em Atkin-Lehner involution} $w_{I_1}$ on
$M_I$ and hence on $H_I$ corresponding to $I_1$.  One can describe $w_I$
in terms of the moduli functor of abelian varieties
with level structure represented by $H_I$.  Namely, a point corresponds
to an abelian surface with real multiplication by $\O_K$ and a subgroup
isomorphic to $\O_K/I$ as an $\O_K$-module, and $w_{I_1}$ replaces
the pair $(A,S)$ with $(A/I_2S,(S+A[I_2])/I_2S)$.

This abstract definition, however, is not very useful for computing the
Atkin-Lehner involution in practice.  Instead we use the description
in terms of modular forms.  The software computes the action on newforms.
For oldforms, the following description, which follows from the
$\GL(2)$ Oldforms Theorem of \cite[p.~4]{roberts-schmidt},
is sufficient for our purposes.

\begin{prop}\label{prop:al-oldforms}
  With notation as above, let $P$ be a prime ideal of $\O_K$,
  let $J$ be an ideal, and let
  $f$ be a Hecke eigenform of level $J$ and weight~$w$.
  Consider the level-raising operator $a_{P^i}$ from forms of
  level $J$ to forms of level $JP^d$, where $0 \le i \le d$.
  Let $w_P$ be the Atkin-Lehner operator associated to the highest
  power of $P$ dividing $JP^d$.  Then
  $w_P(a_{P^i} f) = sp^{(d-2i)(w/2)} a_{P^{d-i}} f$, where
  $s = 1$ if $P \nmid J$ and is equal to the Atkin-Lehner eigenvalue on $f$
  otherwise.
\end{prop}

Thus we may compute the Atkin-Lehner operator on the entire space of modular
forms by expressing the basis given by the software in terms of
Eisenstein series (for which again the action is easily described) and
the images of newforms of various levels under degeneracy maps.  Although
the given basis does not consist of such forms, we may change to a
different basis that does.  There is no ambiguity because we can compute
enough Hecke eigenvalues to determine the expression uniquely.  We
can verify that our expression is correct by checking that it gives
an automorphism of $H_I$.

\subsection{Projective embedding and cusps}\label{subsec:embed}
In computing the Atkin-Lehner operators we determined the relation between
the set of generators of $M_I$ used by \cite{hmf} and bases of Eisenstein
series and cusp forms.  Thus we may identify the cusps of the surface
as the points at which all cusp forms vanish.  (The cusps belong to the
singular locus of $H_I$, which in theory we could find
by the standard Jacobian ideal computation, but this
is infeasible in practice in all but a few cases.)
This gives an additional test for our computation
of Atkin-Lehner operators, since these should permute the cusps.

We now choose an integer $n$ and use it to embed $H_I$ into projective space
by $\O(n)$.  If we are considering the quotient by some or all of the
Atkin-Lehner involutions, then rather than using all of $\O(n)$ we use only
the subspace fixed by the involutions.  In addition, we record the images of
the cusps.

It is much easier to use computer algebra systems to work with a surface if
the singularities of the surface are canonical; in other words, if they
are of ADE type, or Du Val singularities \cite[(1.2)]{ypg}.  However, the cusp
singularities never have this property, and not all elliptic points
\cite[p.~15]{vdg} do.  Thus we would like to resolve these singularities.

Zariski showed that the singularities of a surface can always be resolved
by alternating between normalizing and blowing up isolated singularities.
On a computer, both of these are rather unpleasant
procedures.  As a substitute for blowing up, we use the projection.
That is, let $S \subset \P^n$ be a surface with an isolated singular point $P$.
Projection away from $P$ gives a rational map $S \dashrightarrow \P^{n-1}$,
whose image we will denote by $S_P$.  Abstractly we think of the situation
as follows.  Let $S'$ be the partial resolution of $S$ at $P$ given by blowing
up once, let $H$
be the hyperplane class, and let $E$ be the exceptional divisor.  Then
$S_P$ may be viewed as the embedding of $S'$ by means of the divisor class
$H-E$.

We know $E$ abstractly by the ideas of Hirzebruch \cite[Chapter II]{vdg},
and it is
computed by \cite{hmf} as a cycle of rational curves of known
self-intersection, each one intersecting the adjacent ones.
In particular, this can be used to predict the degree of $S_P$.  This is
valuable, because a K3 surface embedded into $\P^n$ by a complete linear
system always has degree $2n-2$, and so a projection that reduces the
degree by $3$ or more means that we are making progress (unless the
projection itself is not a birational equivalence).

In some cases, after repeated projections away from the singularities, we
obtain a surface $S$ that is singular along a line $L$.  A variety that is
singular in codimension $1$ is not normal \cite[I.4]{kollar}, so naturally we
want to project away from $L$.  The image $S_L$ of $S$ under projection away
from $L$ is in a space of dimension $2$ less and satisfies
$\deg S - \deg S_L \ge 5$, so such an operation always brings us closer to
the degree for a standardly embedded K3 surface.

The projection away
from a singular point or line may be a map of degree greater than $1$.
This is undesirable, because the image may no longer have the desired
cohomology.  When it happens, we apply the $2$-fold Veronese embedding and
proceed with our projections as before.

Even after we have a K3 surface, we may want to make further birational
transformations in order to improve the model.  In particular, the lower the
degree of the singularities of a K3 surface, the easier it is to work with
its linear systems; the tradeoff is that it is more difficult in a
higher-dimensional projective space.  As before, the basic idea is to project
away from the worst singularities and use the Veronese embedding when necessary
to find more room.  Typically it is fairly easy to compute on a model in
$\P^7$ with only $A_1, A_2$ singularities, or in $\P^5$ when there are only
$A_1, A_2, A_3$ points.

\subsection{Tricks for finding images of projections and singular points}\label{subsec:tricks}
In principle the problem of finding the image of a rational map of varieties
can be solved algorithmically by means of an algorithm using Gr\"obner
bases \cite[Sections 3.1--3.3]{CLO}.
In practice this algorithm, as implemented in Magma \cite{magma},
may require more computer time or memory than is available, and so we may
need to find the image in other ways.

First, even if the map is defined by linear equations, as is the case for a
projection map, the implementation does not appear to take advantage of this
but simply treats it as a general map defined by arbitrary polynomials.
However, in this case, the image can be found by calculating an
elimination ideal of the ideal defining the source, rather than of the
graph as in the general method.  To state the result concretely:

\begin{prop}\label{prop:image-from-elim}
  Let $V \subseteq \P^n$ be the subvariety defined by the ideal
  $I \subseteq K[x_0,\dots,x_n]$.  Fix $k \le n$ and let $\pi$ be the
  rational map $\P^n \dashrightarrow \P^k$ defined by $(x_0:\dots:x_k)$.
  Then the Zariski closure of $\pi(V)$ is defined by the elimination ideal
  $K[x_0,\dots,x_k] \cap I$.
\end{prop}

This is almost the definition of the image, but it is a very useful fact since
it is much faster to compute an elimination ideal of $I$ than of the graph
of $\pi$, as is required by the
general algorithm.  (Note that this method can be used even when the map
is defined by linear equations that are not monomials, by composing with
a linear automorphism of $\P^n$.)
However, in some cases even this method fails in practice, or we need to
find the image of a map defined by polynomials of degree greater than $1$.
Then we are forced to resort to another method, which is based on an
even simpler observation:

\begin{lemma}\label{lem:sub-image}
  Let $V$ be a variety and $\pi: V \dashrightarrow W$ a rational map.
  Let $C \subseteq V$ be a Zariski closed subset.  Then
  $\pi(C) \subseteq \pi(V)$.
\end{lemma}

This trivial statement is useful because it is generally much faster to
compute the image of a $0$-dimensional scheme.  Let us state another
proposition that shows that this idea can be powerful.

\begin{prop}\label{prop:compute-low-degree}
  Let $V$ be a variety and $C_0, C_1, \dots$ an infinite sequence of Zariski
  closed subsets of $V$, the Zariski closure of whose union is $V$.  Let
  $\pi: V \dashrightarrow W$ be a dominant rational map, and let
  $W_k = \cup_{i=1}^k \pi(C_i)$.  Then:
  \begin{enumerate}
  \item For all $k$, every polynomial vanishing on $W$ vanishes on $W_k$.
  \item For all $d > 0$, there exists $k_d$ such that every polynomial of
    degree $d$ vanishing on $W_{k_d}$ vanishes on $W$.
  \end{enumerate}
\end{prop}

\begin{proof}
  The first statement holds because $W_k \subseteq W$ by
  Lemma~\ref{lem:sub-image}.
  For the second, let $D_{k,d}$ be the dimension of the space of homogeneous
  polynomials of degree $d$
  vanishing on $W_k$ modulo those in $I(W)$.  Clearly $D_{k,d} \ge D_{k+1,d} \ge 0$
  for all $d, k$.  We show that if $D_{k,d} \ge 0$ then $D_{k+m,d} < D_{k,d}$ for
  some $m$, from which the result follows.  Indeed, let $p$ be a homogeneous
  polynomial of degree $d$ in $I(W_k) \setminus I(W)$.  By hypothesis,
  the Zariski closure of $\cup_{i \in \N} \pi(C_i)$ is $W$, so it is not contained
  in $W \cap (p = 0)$.  Thus for some $m$ we must have $p \notin I(\pi(C_{k+m}))$
  and hence $p \notin I(W_{k+m})$ and $p \notin D_{k+m,d}$.
\end{proof}

This suggests the following empirical approach to finding the image of a map
$\pi: V \dashrightarrow \P^n$:
\begin{alg}\label{alg:find-image}
  \begin{enumerate}\setcounter{enumi}{0}
  \item To start, let $I$ be the empty subscheme of $\P^n$.
  \item\label{item:choose} 
    Choose a linear subspace $L$ of the ambient space of $V$ of codimension
    $\dim V$ and find $\pi(L \cap V)$.
  \item Replace $I$ with $\pi(L \cap V) \cup I$ and observe the degrees of the
    generators of $I$.
  \item If it appears that the low-degree generators have stabilized, then
    guess that $\pi(V)$ is defined by them.  If not, return to
    Step \ref{item:choose}.
  \end{enumerate}
\end{alg}

Of course, this method is incapable of proving that $\pi(V)$ is as we believe
it to be.  But in many cases, once we have a candidate $C$ for $\pi(V)$, we can
verify that it contains $\pi(V)$ (this just means that every generator of the
ideal of our candidate pulls back to an element of $I(V)$).  Often further
information is available that allows us to prove that our candidate $C$
really is equal to $\pi(V)$: for example, the restriction of $\pi$ to a map
$V \to C$ is invertible, or $C$ is irreducible and $\pi(V)$ contains two
distinct codimension-$1$ subvarieties of $C$.

We now discuss how to find the singular points on a projective variety $V$
defined over $\Q$.  Again, this can be done in principle by means of a Jacobian
ideal computation, but when the variety is defined by a large number of
equations in a projective space of high dimension that is impractical.

To illustrate the method in the simplest case, let us suppose that $V$ has
exactly one singular point.  For almost all primes $p$, the reduction of
$V$ mod $p$ will have the same property, and we can find the point mod $p$
by enumerating the $\F_p$-points of the reduction and checking them for
singularity (deciding whether a point is singular is a simple computation
if we know the dimension of $V$).  By looking for rational numbers that
reduce modulo the various $p$ to the coordinates of these points, we guess the
coordinates of the point over $\Q$.  Again, this is not a proof, but we may
easily check whether our candidate is actually on $V$ and, if so, whether
it is singular.

\begin{remark} There is a well-known method for finding $a \in \Q$
  reducing to $a_i \mod p_i$ for primes $p_1, \dots, p_n$  and $a_i \in \F_{p_i}$
  for $1 \le i \le n$.  To wit, let $L$ be the lattice $\Z^2$ and for each
  $i$ consider the sublattice $L_i$ generated by $pL$ and $(a_i,1)$.
  Let $M = \cap_{i=1}^n L_i$.  If $a = b/c$ in lowest terms then $(b,c) \in M$,
  and if $n$ is large enough then it will almost certainly be the shortest
  vector.  Reduction of $2$-dimensional lattices is straightforward, so we
  can find $a$.
\end{remark}

More generally, there might be more than one singular point on $V$ and they
might not be rational.  This may cause problems because we would not know
which collections of reductions come from the same singular point; even if
there are only $2$ singular points trying all the possibilities leads to
a combinatorial explosion.  In some cases the points will have
some properties that make it clear which reductions go together,
such as some conspicuous
$0$ coordinates, or we will already know some singular points.  
Otherwise we may have to consider the reduced $0$-dimensional schemes supported
at the singular points mod $p$ and interpolate the coefficients in the
Gr\"obner basis of the ideals defining them.

If the points are not rational, then for some primes we will only
be able to find their reductions by passing to an extension of $\F_p$, which
is undesirable.  Fortunately in this work it was rarely necessary to find
singular points not defined over a quadratic field.  This case can be
recognized by the singular points appearing on the reductions modulo half of
the primes.  There are a few more complicated cases with singular points
defined over different fields; however, the quadratic fields of definition
are easily recognized and so one can guess a formula for the number of
singular points over $\F_p$ in terms of quadratic residue symbols.  By
considering the reductions modulo primes for which exactly one of the symbols
is $1$ one can find the points.

\section{Verification of modularity}\label{sec:verify-modularity}
Let us now suppose that we have found two K3 surfaces: a quotient
$K_0$ of the Hilbert modular surface $H_I$ and the Kummer surface $K_1$ of
$E \times E^\sigma$.  We would like to show that these are in correspondence,
which in practical terms means that there is a sequence of K3 surfaces
$K_0 = L_0, L_1, \dots, L_n = K_1$ and, for all $0 < i < n$, a rational map
of finite degree $L_i \to L_{i+1}$ or $L_{i+1} \to L_i$.  Essentially we know of
four ways to construct rational maps, three of which are based on genus-$1$
fibrations.

Let $S$ be a K3 surface defined over $\Q$, and let
$\pi: S \to \P^1$ be a map whose general fibre is a smooth curve of genus $1$.

\begin{notation}\label{not:pL}
  Let $L$ be a lattice.  For $r \in \Q^+$ let $rL$ be the sublattice of
  $L \otimes \Q$ whose vectors are the $rv$ for $v \in L$, and let
  $L(r)$ be the lattice whose underlying abelian group is that of $L$
  but with the pairing $(x,y)_{L(r)} = r(x,y)_L$.  (To clarify this notation,
  we point out that $rL$ and $L(r^2)$ are isometric.)
\end{notation}

\begin{enumerate}
\item\label{item:isogeny}
  Let $p$ be a prime and suppose that the $p$-torsion of the general fibre
  of $\pi$ has a nontrivial subgroup $T$ of order $p$ defined over $\Q(t)$.
  In this case we may define an isogeny on the general fibre of $\pi$
  by quotienting by $T$; this spreads out to a map from $S$ to another
  K3 surface.  The effect
  of this on the transcendental lattice $T_S = \HHH^2(S)/\Pic S_{\bar \Q}$
  of $S$ is to replace it by a
  lattice between $T_S(1/p)$ and $T_S(p)$.
  (This exact statement is not easy to find in the literature.
  It follows from remarks in \cite[Section 2.4]{bsv} that a map
  $X \dashrightarrow Y$ of K3 surfaces of degree $n$ induces an embedding
  $T_Y(n) \hookrightarrow T_X$, whence the given statement follows from the
  existence of a dual isogeny.)
\item\label{item:jacobian}
  If $\pi$ has no sections, then let $d$ be the smallest degree of a
  multisection (which is the GCD of the degrees of intersections of curves
  on $S$ with a fibre of $\pi$).  The map from the generic fibre to its
  Jacobian spreads out to a map of degree $d^2$ from $S$ to another K3 surface
  $\Jac_\pi(S)$.   If $E$ is the class of a fibre of $\pi$, then
  we have $\Pic \Jac_\pi(S) = \Pic S[E/d]$ \cite[Lemma 2.1]{keum},
  at least after
  base change to $\bar \Q$ (in other words, this is not an isomorphism of
  Galois module structures).
\item\label{item:q-jacobian}
  It can happen that $\pi$ has no sections over $\Q$ but does over a
  larger number field $K$.  In this case $\Jac_\pi(S)$ is $K$-isomorphic to
  $S$ and so we have $\Pic S_{/K} \cong \Pic \Jac_\pi(S)_{/K}$.  However,
  the submodule of $\Pic \Jac_\pi(S)$ fixed by $\Gal(\Q)$
  is larger than for $S$.  This can be a useful operation because it enlarges
  the range of fibrations available to us, thus creating new maps of types
  (\ref{item:isogeny}), (\ref{item:jacobian}).
\item\label{item:abvar}
  If $S$ is a quotient of an abelian surface $A$ by a finite group $G$ of
  automorphisms and $\phi: A \to A'$ is an isogeny whose kernel is
  $G$-invariant, then $\phi$ descends to a rational map $S \to S'$.
  In particular this applies if $S$ is a Kummer surface, the condition
  on $\ker \phi$ being automatic in that case.  We
  can understand the relation between the transcendental lattices of $S$
  and $S'$ induced by such a map from that between $T_A, T_{A'}$, since
  $T_S = T_A(2)$ and $T_{S'} = T_{A'}(2)$.
\end{enumerate}

\begin{remark} Note that in our situation
  the geometric Picard rank of $K_1$ is at least $18$,
  so there is always an abundance of genus-$1$ fibrations.
\end{remark}

\begin{remark} The geometric Picard rank of a K3 surface is invariant under
  finite maps \cite[Theorem 1.1]{bsv}.  There are additional invariants in the
  cases of rank $18$ and $19$ that we are concerned with here.  Namely,
  when the rank is $18$ (or any even integer), then the class of the
  discriminant of the Picard lattice in $\Q^*/(\Q^*)^2$ is invariant under
  all three types of maps described above.  When the rank is $19$, the
  intersection form on $T_S$ defines a conic in $\P^2$, whose
  $\Q$-isomorphism class is invariant under the three types of maps.
\end{remark}

\begin{remark} Let $K_0, K_1$ be K3 surfaces for which we would like to find
  a correspondence.  Let the $D_i$ be the discriminants of the respective
  geometric Picard lattices, and let $P$ be the set of primes such that
  $v_p(D_0/D_1)$ is odd.  If the rank is even and $P$ is nonempty, the task
  is hopeless; if the rank is odd, we start by finding a $p$-isogeny for each
  prime dividing $p$ (if this is not possible, the task again seems hopeless).
  Thus we assume that $P$ is empty.

  In general it is easier to work with K3 surfaces that have a large part of
  their Picard lattice defined over $\Q$; this makes it easier to find
  elliptic fibrations and maps as defined above.  This generally goes with
  a smaller Picard discriminant.  Thus we want to apply maps that
  reduce the discriminant if possible; for isogenies this can be difficult
  to predict, but it always happens when we take a Jacobian.  On the other hand,
  if we cannot find any fibrations without a section defined over $\Q$, it may
  still be possible to describe a fibration with a section but none over $\Q$
  and apply a step of type \ref{item:q-jacobian}, which then allows us to
  continue by a step of type \ref{item:isogeny} or \ref{item:jacobian}.

  Another advantage of reducing the discriminant is that there are only a
  few lattices of small discriminant, and each one supports only a small number
  of frame lattices for elliptic fibrations \cite[(8.6)]{schutt-shioda}.
  This makes it easier to recognize the isomorphism of K3 surfaces.
\end{remark}

\begin{remark}\label{rem:reduce-weierstrass}
  Let us suppose that we have found a Weierstrass equation for
  the general fibre of an elliptic fibration.  It may be defined by polynomials
  with unreasonably large coefficients, causing future computations to be slow.
  This situation can often be remedied as follows: find a minimal integral
  model and let $D(t) = \prod_i t-\alpha_i$, where the $\alpha_i$ are the
  coordinates of singular fibres excluding $\infty$.
  Find an automorphism of $\P^1$
  under which the numerator of the pullback of $D$ has smaller coefficients.
  (In Magma, this amounts to minimizing and reducing the hyperelliptic curve
  defined by $y^2 = D(t)$.)  Changing coordinates on the base by means of the
  same linear transformation often results in a much more usable equation,
  and twisting by a large square common factor of the coefficients is a
  further improvement.
\end{remark}

\section{Examples}\label{sec:ex}
\subsection{Summary}
In this paper we are mainly concerned with the cases where $p_g(H_{I,K}) = 1$.
Up to Galois conjugacy, Hamahata \cite{hamahata} lists $14, 6, 4, 2$ such
surfaces over the real quadratic fields of discriminant $5, 8, 13, 17$
respectively.  (Such surfaces also exist over fields whose strict class
number is not $1$, and over the fields of discriminant $29, 37, 41$ at level
$1$, but these are not his concern.)  We summarize his results on the
Kodaira dimensions.

\begin{table}[h!]\label{tab:kod-dims}
  \caption{Kodaira dimensions of Hilbert modular surfaces with geometric
    genus $1$ over real quadratic fields with narrow class number $1$.}
  \begin{tabular}{|c|c|c|c|c|}
  \hline
  &$\kappa=0$&$\kappa=1$&$\kappa=2$&$\kappa \in \{1,2\}$\\ \hline
  $D=5$&0&0&9&5\\  \hline 
  $D=8$&1&2&2&1\\  \hline 
  $D=13$&1&1&2&0\\ \hline 
  $D=17$&2&0&0&0\\ \hline 
\end{tabular}
\end{table}

\begin{remark}\label{rem:hamahata-error}
  In this table we have corrected an error in \cite{hamahata}:
  as we will see in Section \ref{ex:d17-rt2},
  the surface $H_{\p_{17},\Q(\sqrt{2})}$ has Kodaira dimension $1$.
  This means that the argument in \cite[6.9]{hamahata} is invalid;
  in fact Hamahata's $d_m(4;1,1)$ is approximately $m^2/2$, so that the lower
  bound for $P_m$ is $O(m)$.  I thank Eran Assaf for explaining this to me.
\end{remark}

\begin{remark} The cases of ambiguous Kodaira dimension from \cite{hamahata}
  can all be resolved.  In particular, let us define the
  {\em standard model} of a Hilbert modular surface (a standard concept,
  but not a standard term) to be that obtained by resolving all
  singularities and then recursively contracting all
  Hirzebruch-Zagier cycles and components of resolutions
  of self-intersection $-1$.  If the Kodaira
  dimension is nonnegative, then no model has intersecting $-1$-curves,
  so this is unambiguous.  Such a model is
  believed to be minimal, but this is not known in general.  

  The software \cite{hmf} can compute the self-intersection of the
  canonical divisor on the standard model.  In 
  five of the six cases left undecided by Hamahata,
  namely $(6), 3\p_5, \p_5\p_{11}, 4\p_5$ over
  $\Q(\sqrt{5})$ and $2\p_7$ over $\Q(\sqrt{2})$, it is positive, so
  the Kodaira dimension cannot be $0$ or $1$ and must be $2$ since
  $\hh^{2,0} \ne 0$.  On the other hand, the self-intersection is $0$
  for $\p_{31}$ over $\Q(\sqrt{5})$.  Again, this does not show that the surface
  is not of general type, because the model is not known to be minimal.
  We will discuss this example in
  detail in Section~\ref{ex:d5-p31}, showing that the Kodaira dimension is $1$.
\end{remark}

\begin{remark}
  The surfaces of general type are extremely interesting, but they seem
  to be very difficult for the methods of this paper.  As mentioned in the
  introduction, a general theorem of Morrison \cite{morrison} asserts the
  existence of a correspondence between a surface with $\hh^{2,0} = 1$ and
  a K3 surface.  However, except in some special cases such as Todorov
  surfaces (double covers of K3 surfaces that were first studied as examples
  of surfaces of general type that do not satisfy a Torelli theorem),
  the correspondence is very hard to find.  Some of these
  surfaces admit Atkin-Lehner quotients that are not of general type; instead
  they are elliptic surfaces of a very interesting kind.  This applies to
  $H_{\p_{19},\Q(\sqrt{5})}$, for example.

  Hamahata also studies the surfaces with $p_g = 0$.  Since these are not
  associated with any modular forms they are not relevant for the
  Oda-Hamahata conjecture,
  but many of them are of general type and are thus relevant to the important
  tradition of constructing surfaces of general type with $p_a = p_g = 0$ and
  small $K^2$.  We intend to study these two situations in more detail in
  future work.

  Of the $7$ cases with $p_g = 1$
  where the surface was known not to be of general type,
  there are only $2$ that do not correspond to $\Q$-curves, so that the
  conjecture is not already known.  As mentioned above,
  the surface $H_{\p_{31},\Q(\sqrt{5})}$ has Kodaira dimension $1$,
  and we will show  that
  $H_{2\p_7,\Q(\sqrt{2})}$, another surface left undecided by Hamahata, is of
  general type.  These and other examples will be discussed in
  Section \ref{sec:general-type}.
\end{remark}

\begin{remark} There are some cases in which $H_{I,K}$ has $p_g = 0$ but is
  nevertheless related to a K3 surface.  For example, if $I = (3)$ and
  $K = \Q(\sqrt{5})$, then $H_{I,K}$ is an Enriques surface and its unramified
  double cover is a K3 surface of Picard rank $20$ and discriminant $-240$.
  Thus it is related to the Kummer surface of $E \times E^\sigma$, where
  $E$ is an elliptic curve with complex multiplication by $\sqrt{-15}$,
  in particular a curve in the class \elllmfdb{2.2.5.1}{81.1}{a} with
  conductor $(9)$.  
\end{remark}

\begin{defn}\label{def:diameter}
  Let $\isi$ be an isogeny class of elliptic curves
  (over a number field, or otherwise constrained to be finite).
  If $E_1, E_2 \in \isi$, let $d(E_1, E_2)$ be the minimal degree of
  an isogeny from $E_1$ to $E_2$.  (Note that the distance function $\log d$
  makes $\isi$ into a metric space.)  The {\em diameter}
  $\dia(\isi)$ is defined to be $\max_{E_1,E_2 \in \isi} d(E_1,E_2)$.
\end{defn}
  
The Picard group of the K3 surfaces that are
quotients of the Hilbert modular surface may be of interest.
Empirically the following statement seems to hold:

\begin{claim}\label{claim:disc-pic}
  Let $H_{I,K}$ be a Hilbert modular surface and let $S$ be a K3 surface
  which is a quotient of $H_{I,K}$ not factoring through any other map to
  a K3 surface and associated to an isogeny class $\isi$.  
  The Picard discriminant of $S$ is a very smooth number
  (usually a small power of $2$) times $\dia(\isi)$.
\end{claim}

In particular, if $\dia(\isi)$ is divisible by a prime greater than $5$, we
expect it to be very difficult to find a correspondence between $S$ and a
Kummer surface.

We now proceed to a detailed examination of some individual examples.
The code that supports the claims made in the examples of this and the
next section is available at \cite{code}.

\subsection{Level $\p_2^5$ over $\Q(\sqrt{2})$}\label{ex:dp25-rt2}
Our first example is very much simplified by the fact that the associated
elliptic curve is not merely a $\Q$-curve but in fact the base change of
an elliptic curve over $\Q$ with complex multiplication.  In particular,
the elliptic curve $y^2 = x^3-x$ has conductor $\p_2^5$ over $\Q(\sqrt{2})$,
with the reduction at $\p_2$ being of type $I_3^*$.  However, let us begin
with the Hilbert modular surface.  We calculate that the ring of Hilbert
modular forms is generated by $9$ forms of weight~$2$ and $1$ of weight~$4$.
The number $9$ comes from the $8$ cusps giving $8$ Eisenstein series and
the $1$ cusp form corresponding to the elliptic curve.

According to \cite[6.11]{hamahata} the Hilbert modular surface
$H_{\p_2^5,\Q(\sqrt{2})}$ is an honestly elliptic surface (i.e., not a K3
or rational elliptic surface).  We will go directly to the
quotient by the Atkin-Lehner involution.  Since the space of weight-$4$ cusp
forms with Atkin-Lehner eigenvalue $-1$ is contained in the space of
products of weight-$2$ modular forms, we need a form in the $+1$ eigenspace,
meaning that the quotient is naturally in weighted projective space
$\P(1^5,2)$.  The quotient surface $S$  has $4$ cusps.

\begin{prop}\label{prop:first-map-p25}
  The image of $S$ by the map given by forms of degree $2$ vanishing on the
  cusps modulo those in the ideal of $S$ is a surface $S_9$ of degree $24$ in
  $\P^9$ in which the cusps resolve to singular lines.
\end{prop}

\begin{proof} This is calculated by means of Algorithm~\ref{alg:find-image}.
\end{proof}

Projection away from any one of these lines gives a surface
of degree $18$ in $\P^7$; since they are disjoint, one heuristically expects
projection from the $\P^7$ that they span to give a
``surface of degree $0$ in $\P^1$''.  This can be interpreted as
a K3 surface with a genus-$1$ fibration, which prompts the next proposition.

\begin{prop}\label{prop:s9-fib}
  The map $S_9 \to \P^1$ given by the linear forms vanishing on all four
  singular lines has fibres of degree $16$ and arithmetic and geometric
  genus $1$.
\end{prop}

\begin{proof} Once guessed this is easily checked in Magma.  See \cite{code}.
\end{proof}

By projecting the generic fibre away from 
the union of three of the four degree-$4$ divisors
obtained by intersecting it with lines we find a quartic model
in $\P^2$ singular at two points $P_1, P_2$.  The linear system of quadrics
vanishing at $P_1, P_2$ gives a map to $\P^3$ whose image is a smooth
intersection of two quadrics in $\P^3$ by
the map given by quadrics in $\P^2$ vanishing on the singular points.
Taking the Jacobian of the generic fibre we define an elliptic curve $\E$ over
$\Q(t)$ that spreads out to a K3 surface.

\begin{remark}\label{rem:multidegree-4}
  The generic fibre has no $\Q(t)$-rational divisors of degree
  between $0$ and $4$, because there are fibres with no points over any
  quadratic extension of $\Q_3$.
  We do not know whether the generic fibre admits any
  $\bar \Q(t)$-rational divisors of degree between $0$ and $4$,
  but we suspect not.
\end{remark}

In fact, this elliptic curve has a simple description.  Up to changes of
coordinates on the base, it is the unique elliptic curve over $\P^1$
with fibres of types ${I_1^*}^2, I_8, I_2$; the torsion subgroup has order $4$,
the rank must be $0$ since the known curves already account for Picard rank
$20$, and so the Picard discriminant is $-16$.  Applying the method of
Remark~\ref{rem:reduce-weierstrass} we compute the equation
$$y^2 = x^3 + (-2t^3 + 12t^2 - 2t)x^2 + (t^6 + 4t^5 + 6t^4 + 4t^3 + t^2)x.$$
Over $\bar \Q$ the general theory of K3 surfaces of rank 20
\cite{SI}
can easily be used to show that such a surface is in correspondence with
$\Kum(E \times E)$, where $E: y^2 = x^3 - x$ is an elliptic curve with
complex multiplication by $\Z[i]$.  In particular, one checks
that the transcendental lattice is isomorphic to that of the lattice
${\langle 4 \rangle}^2$ generated by two orthogonal vectors of norm $4$,
and the result follows by \cite[Section 3]{shioda-mitani}.
However, over $\Q$ our work is not quite done.

We remark that all components of reducible fibres are rational.  Indeed,
the rational $4$-torsion means that every component of a reducible fibre that
meets a multiple of the $4$-torsion section must be rational, and none of
the fibres admits any nontrivial permutation of its components preserving
intersections and fixing all of these curves.  On the other hand, the points
of intersection of the components of the $I_2$ are not rational, being
instead defined over $\Q(i)$.

Counting points mod $p$ on $\E$, we find results consistent with the existence
of a correspondence to $\Kum(E \times E)$, where as above $E$ is defined by
$y^2 = x^3 - x$.

\begin{prop}\label{prop:fib-kum-ee}
  There is a fibration on $\Kum(E \times E)$ whose general fibre
  is isomorphic to the elliptic curve
  $E'$ defined by $y^2 = x^3 + (t^3 + t)x^2 - 4t^4x + (-4t^7 - 4t^5)$.
\end{prop}

\begin{proof}
We write down a model of the Kummer surface
in $\P^4$ and use it to determine a fibration with two fibres each of type
$I_2^*$ and $I_4$.
Concretely, our model is calculated by
  mapping $E \times E \subset \P^2 \times \P^2$ to $\P^4$ by sections of
  $\O(1,1)$ invariant under the product of the two negation maps.  It has
  an $A_3$ singularity which is the image of
  $E \times 0 \cup \{(0,0)\} \cup 0 \times E$ and nine $A_1$ singularities
  under the points $(T_1,T_2)$, where $T_1, T_2$ are points of $E$ of exact
  order $2$.

  The fibration we seek has fibres of class
  $H - E_1 - E_2 - E_3 - L_1 - L_2$, where the $E_i$ are the components of the
  resolution of the $A_3$ and $L_1, L_2$ are the images of
  $(O,(0:0:1)), ((0:0:1),O)  \in E \times E$ (more precisely, of the
  curves above these points when the indeterminacy of the rational map
  $E \times E \dashrightarrow \P^4$ is resolved).  It is a straightforward
  calculation in Magma to determine the general fibre of this map and its
  Jacobian.
\end{proof}

At this point, it only remains to verify that, up to change of coordinates
on $\P^1$, the elliptic curves $E, E'$ are isomorphic.  It is easy to find
the correct change of coordinates, because the bad fibres must be in the
same places for the curves to be isomorphic.

\subsection{Level $(3)$ over $\Q(\sqrt{13})$}\label{ex:d3-rt13}
We now present an example that, although still relatively simple, is
more complicated than that of Section~\ref{ex:dp25-rt2} and
illustrates many of the methods of calculation used in this paper.
Namely, we consider the isogeny class \elllmfdb{2.2.13.1}{9.1}{a}
of elliptic curves of conductor $(3)$ over $\Q(\sqrt{13})$.  Again,
we have a $\Q$-curve and so the Oda-Hamahata
conjecture is already known.
Nevertheless, the explicit calculation is of some interest.

Throughout this section, let $K = \Q(\sqrt{13})$ and let $I = (3)$.
According to Hamahata \cite{hamahata}, the Hilbert modular surface $H_{I,K}$ has
Kodaira dimension $1$ and hence has a unique genus-$1$ fibration.
It would be possible to find this fibration;
the Jacobian would be the desired K3 surface.  Instead we begin by determining
the Atkin-Lehner involutions.  Since $3$ splits in $K$, there are two.

\begin{prop}\label{prop:deg-24-p-10}
The subspace of $\O(3)$ (i.e., the space of modular forms of weight~$6$)
fixed by both Atkin-Lehner involutions modulo forms vanishing on the
Hilbert modular surface has dimension $11$.  A basis defines a map to
$\P^{10}$ whose image is a surface $S_{10}$ of degree $24$.
The four cusps of $H_{I,K}$ are
permuted transitively by the Atkin-Lehner involutions and hence map to a
single singular point of $S_{10}$.
\end{prop}

\begin{proof} This is computed in Magma with the help of
  Algorithm~\ref{alg:find-image} as has already been explained.  For the
  last statement, we note that since the level is squarefree the cusps are
  in natural bijection with the divisors of $N$, and the Atkin-Lehner
  operator $w_I$ takes the cusp $c_J$ to $w_{IJ}$ if $IJ|N$ or $w_{IJ/N}$ if
  not, just as in the classical setting.
\end{proof}

We find the cusps on $H_{I,K}$ as the points where all cusp forms vanish.
The cusp on $S_{10}$ is resolved by a cycle of rational curves of degrees
$-5,-2,-2$ \cite{hmf}.  Let these curves be $E_1,E_2,E_3$.  Then
$(\O(3)-(E_1+E_2+E_3))^2 = 24 + (-2-2-5+3\cdot 2) = 21$, since $(\O(3),E_i) = 0$
for $i = 1, 2, 3$.  Thus, when we project away from the cusp, we obtain a
surface $S_9$ of degree $21$.

\begin{prop}\label{prop:e9-c9}
  The surface $S_9$ has at least two singular points, including the
  point $E_9$ which is the image of the elliptic points of $H_{I,K}$
  by the map to $S_{10}$ followed by projection away from the cusp and
  the point $C_9$ which is in the image of the cusp.
\end{prop}

\begin{prop} Let $S_8$ be the projection of $S_9$ away from $C_9$ and let
  $C_8$ be the image of $C_9$ in $S_8$.  For $i = 8, 7, 6$, inductively define
  $S_{i-1}$ to be the projection of $S_i$ away from $C_i$ and $C_{i-1}$ to
  be the unique singular point in the image of $C_i$.  Then $S_{i-1}$ is a
  surface in $\P^{i-1}$ of degree $2i-1$.
\end{prop}

\begin{prop}\label{prop:blowup}
  $S_5$ is the blowup of a K3 surface in a single point.
\end{prop}

\begin{proof} Using built-in Magma functions for computing canonical models
  we find a map to a K3 surface
  of degree $10$ in $\P^6$ whose inverse fails to be defined at a single point.
  The inverse image of that point can easily be checked to be a smooth curve
  of self-intersection $-1$.
\end{proof}

This K3 surface $S_{10}$ of degree $10$ in $\P^6$ is defined by $6$ quadrics
as expected.  By good fortune all of its singularities are of
type $A_1$; there are two orbits over $\Q(\zeta_{12})$, one over
$\Q(\sqrt{13})$, and two over $\Q$, making $12$ singular points in all.

In order to study $S_{10}$ more closely, we find some curves on it.  One
way to do this is to note that one of the $\Q$-rational singular points has
the property that a line through it and one of the two singular points
defined and conjugate over $\Q(\sqrt{13})$ lies on the surface.  Another is
to use the genus-$1$ fibrations defined by hyperplanes passing through
one orbit of $\Q(\zeta_{12})$-rational singular points and one rational
singular point.  These are all fibrations with three $I_3$ fibres and two
of type $I_4$, and so each one gives us $17$ rational curves contained in
fibres.  However, some of these are the curves in the resolutions of the
$A_1$ singularities, and the overlap means that we have only found $12$ curves
in total.

This is enough for now.  We have a sublattice of the Picard group with
$18$ generators (later we will find a $19$th on a related K3 surface).

\begin{prop} There is a genus-$1$ fibration on $S_{10}$ defined over $\Q$
  whose Jacobian is defined by
\begin{align*}
  y^2 &= x^3 + (-8t^4 - 40t^3 - 2t^2 + 648t + 150)x^2 + (16t^8 + 160t^7 + \\
  &\qquad 408t^6 - 2680t^5 - 15607t^4 - 18240t^3 + 53626t^2 - 31400t + 5625)x.
\end{align*}
\end{prop}

\begin{proof} Using our knowledge of the Picard group of $S_{10}$ and the
  action of Galois on it, we identify a certain
  fibration with fibres of degree $4$.
  The individual reducible fibres of this fibration are
  only defined over $\Q(i)$, so we start by writing down equations for the
  fibration over $\Q(i)$ and then move to $\Q$ by a standard Galois descent.

  Since the fibres are of degree $4$, Magma has no trouble determining the
  Weierstrass model of the Jacobian.  Using the method of
  Remark~\ref{rem:reduce-weierstrass}, we compute the model above.
\end{proof}

This is an elliptic surface $\E$ with two singular fibres of type $I_0^*$ and
one each of types $I_6, I_3, I_2, I_1$.  (It is therefore a quadratic twist
of the second family on \cite[p.~336]{herfurtner},
which is the universal elliptic curve
with a $6$-torsion point.)  Over $\Q(\sqrt{13})$ there is a section with
$x$-coordinate $(13/9)\cdot (t^2+8t+25)\cdot (2t-1)^2$; its height is
$2/3$, since it passes through nonzero components of the two $I_0^*$ fibres
and a component of the $I_6$ at distance $2$ from the zero section
(cf.~\cite[Table 1.19, Theorem 1.33]{cox-zucker}).
We calculate that the torsion subgroup has order
$2$, so the determinant formula \cite[Corollary 6.39]{schutt-shioda}
shows that the discriminant of the Picard
lattice is $4^2 \cdot 6 \cdot 3 \cdot 2 \cdot 2/3 \cdot 1/4 = 96$.  This is
not exactly what we are looking for, since our isogeny class contains
elliptic curves $3$-isogenous to their Galois conjugates.

Since we will need genus-$1$ fibrations on the K3 surface, we must now
find a good projective model of $\E$.  Applying the techniques of
Section~\ref{subsec:embed} we rapidly come to a model $S_3$ in $\P^3$ with
$9$ ordinary double points and no other singularities.

\begin{prop}\label{prop:d6-d4s-a1-fib}
  $S_3$ admits a fibration with no section,
  singular fibre types $I_2^* {I_0^*}^2 I_2^2$, full level-$2$ structure,
  and fibres of degree $6$.  The Jacobian admits a $2$-isogeny to
  the elliptic surface $\E'$ defined by
  $$y^2 = x^3 + (-81t^3 - 95t^2 + 176t)x^2 + (256t^4 - 512t^3 + 256t^2)x,$$
  which has a point with $x$-coordinate $(4t+9t^2/2)^2$ and height $3$.
\end{prop}

\begin{proof} Again this is done computationally.  The only slightly tricky
  point is that, though the generic fibre has degree $6$ and arithmetic
  genus $2$, the projection away from the singular point is a map to a
  rational curve, not to a singular quartic.  Thus, to determine the
  Weierstrass model, we need to start by using the 
  fibre and its Veronese embedding to place the fibre in $\P^9$.
  This done, we can project away from singular points until we reach
  an intersection of two quadrics in $\P^3$.
\end{proof}

\begin{remark}
  The fibration on $S_3$ was chosen precisely because it is 
  $2$-isogenous to a fibration with $I_4^* {I_0^*}^2 I_1^2$ fibres,
  which matches one of the types of fibration on the Kummer surface of
  $E \times E'$.  Thus we know that, at least over $\bar \Q$, the
  surface $\E'$ is isomorphic the Kummer surface of a product of two
  elliptic curves.
\end{remark}

In order to facilitate the comparison between $\E$ and the Kummer surface,
we now find the fibration with reducible fibre types $I_6^* I_2^6$.
This is again done as in Section~\ref{subsec:embed}: we find a model in
$\P^5$ with at worst $A_3$ singularities and find such a fibration there.
We conclude:

\begin{prop}\label{prop:s3-d10-a1s}
  There is a fibration on $S_3$ whose generic fibre has
  reducible fibres $I_6^*, I_2^6$ which is defined by the equation
  $$y^2 = x^3 + (6t^3 - 119/4t^2 - 128t + 
        147/4)x^2 + (-54t^3 + 1071/4t^2 + 1152t - 1647/4)x.$$
\end{prop}

We now start again from the other side, with the Kummer surface of
$E \times E'$ where $E, E'$ are generic elliptic curves with full level-$2$
structure defined by equations $y^2 = x(x-1)(x-p)$ and $y^2 = x(x-1)(x-p')$
where $p, p'$ are independent transcendentals.  (These calculations use
code originally written for \cite{shtukas} and my coauthor Jared Weinstein
had considerable influence on it.)
We find the equation of a fibration with $I_6^* I_2^6$ fibres and identify
$p$ and $q$ by matching the locations of the $I_2$ fibres.  In particular,
we choose $p, p' = -87 \pm 24 \sqrt{13}$ to find two elliptic curves
$E, E'$ of conductor $(3)$
over $\Q(\sqrt{13})$ that are $3$-isogenous.  It is easy to write down a
model over $\Q$ for the Kummer surface of the product and the $I_6^* I_2^6$
fibration on this surface.  However, 
it cannot be isomorphic to that on the other one,
because the dimensions of the Galois-fixed subspaces are different: $11$ on
the Kummer surface and $16$ on the image of the Hilbert modular surface.
As before, we need to take the Jacobian of a fibration with no rational
sections.  We state the result as a proposition.

\begin{prop}\label{prop:kum-d10-a1s}
  There is a fibration on the Kummer surface of $E \times E'$
  defined over $\Q$ that has a $\bar \Q$-section but no rational section.
  Its Jacobian is defined by the Weierstrass equation
  $$y^2 = x^3 + (6t^3 - (119/4)t^2 - 128t + (147/4))x^2 + 
  (-54t^3 + (1071/4)t^2 + 1152t - (1647/4))x.$$
\end{prop}

\begin{remark} It was sensible to look for this particular fibration, because
  taking its Jacobian over $\Q$ increased the rank of the rational part of
  the Picard group substantially.  An $I_6^*$ fibre may only contribute
  rank $5$ to the rational part of the Picard lattice when there is no
  rational section, but must contribute at least $10$ when there is one,
  because the rational section forces the component it meets to be rational,
  and therefore all components except for the two farthest from the zero
  component to be.

  There is an ambiguity when we speak of $\K$ as a K3 surface over $\Q$, since
  we could twist by $13$ to obtain another K3 surface whose base change
  to $\Q(\sqrt{13})$ is isomorphic to that of $\K$.
  It is easy to see which one must 
  be correct, since the number of points mod $p$ must match mod $p$ for
  all primes $p$ of good reduction.
\end{remark}

We observe that the Weierstrass equations in Propositions
\ref{prop:s3-d10-a1s} and \ref{prop:kum-d10-a1s} are identical.
The following theorem summarizes the content of this section:
\begin{thm}\label{thm:13-9} Let $H_{I,K}$ be the Hilbert modular surface
  of level $(3)$ over $\Q(\sqrt{13})$ and let $\K$ be the Kummer surface
  of $E \times E^\sigma$, where $E$ is the elliptic curve defined by
  $y^2 = x(x-1)(x-(87+24\sqrt{13}))$.  Then there is a K3 surface $S$
  with maps of degree $32, 4$ from $H_{I,K}$ and $\K$ respectively,
  such that the correspondence associated to these maps induces an isogeny
  of the transcendental lattices of $H_{I,K}$ and $\K$.
\end{thm}

In other words, we have made the Oda-Hamahata
conjecture completely explicit in this case.

\subsection{Level $\p_2^3$ over $\Q(\sqrt{17})$}\label{ex:dp23-rt17}
We now consider the case of level $\p_2^3$ over $\Q(\sqrt{17})$, where $\p_2$
is a prime of norm $2$.  In this case the level is not Galois-stable and
so we do not have a $\Q$-curve; the conjecture was
not already known in this case.

In this example we choose not to use the Atkin-Lehner quotient.
The ring of Hilbert modular forms is generated by $5$ forms of weight
$2$ and $7$ of weight~$4$ and thus its Proj embeds into $\P(1^5,2^7)$ and hence
into $\P^{21}$ by $\O(2)$.  We find that the image $S_{21}$
is defined by $151$ quadrics and is of degree $64$.

\begin{prop}\label{prop:bir-to-k3}
  There is a birational map from $S_{21}$ to a surface $S_5'$ in $\P^5$
  defined by $3$ quadrics and whose only singularities are $6$ ordinary
  double points.
\end{prop}

\begin{proof}
We can find $4$ singular points on $S_{21}$, as we should since there are $4$
cusps for $\Gamma_0(8)$ over $\Z$ and no elliptic points;
the projection away from these gives a map to $\P^{17}$ whose image $S_{17}$ is 
a surface of degree $44$ which again has $4$ singular points.  Projecting away
from these we now have a surface $S_{13}$ of degree $36$ in $\P^{13}$ and
yet again $4$ singular points.  This time the projection goes to
$S_9 \subset \P^9$, which has degree $24$ and (at least) $7$ singular points,
and by a careful choice of which four to project from we obtain a surface
$S_5$ of degree $12$ in $\P^5$.  A final projection from $S_5$ yields a surface
$S_4$ of degree $10$ in $\P^4$.

This surface in $\P^4$, in turn, has singularities along two lines and three
points (pairwise disjoint),
and the linear system of quadrics vanishing on these gives a map to a different
singular model $S_4' \subset \P^4$, this time of degree $8$.
The linear system of quadrics vanishing on the
two singular lines and two singular points of this surface gives
a map to a surface $S_5' \subset \P^5$ as described in the statement.
\end{proof}

The surface $S_5'$ is a K3 surface.  
One way to find genus-$1$ fibrations on a K3 surface in $\P^n$ is to consider
the families of hyperplanes through $n-1$ singular points.  In this case it
turns out that one of the $\binom{6}{4}$ sets of singular points is coplanar,
but the other $14$ do indeed give genus-$1$ fibrations.  In no case does the
degree of the fibre exceed $4$, so Magma's built-in functions easily find the
Jacobians of these fibrations and hence their reducible fibres.  It turns out
that all $44$ components of these are defined over $\Q(\sqrt{17})$ and that,
together with the curves
above the nodes, they generate a sublattice of the Picard lattice of rank $18$
and discriminant $-64$.  We can verify that the rank is $18$ by reducing modulo
small primes and finding an irreducible factor of degree $4$ in the
$L$-function.  We also verify that we have found the full Picard lattice
as follows.

\begin{prop}\label{prop:disc-more-than-16}
  Let $S$ be a K3 surface with an elliptic fibration $\pi$ with
  reducible fibres $A_7, A_1^8$ and rank $1$.  Then the discriminant of
  $\Pic S$ cannot be $1$, $4$, or $16$.
\end{prop}

\begin{proof}
  According to the tables of \cite{shimada}, these reducible fibres
  force the fibration to have full $2$-torsion, which means that if the
  generator $s$ of the Mordell-Weil group modulo torsion has height $h$ then
  the discriminant is $(8 \cdot 2^8 \cdot h)/2^4$.  Thus we would need
  $h \in \{1/128,1/32,1/8\}$.  By the results of \cite{cox-zucker}, the first
  two are impossible: the height is the sum of an integer with numbers of
  the form $a(8-a)/8, b(2-b)/2$, so there is no way for the denominator
  to be greater than $8$.  There are several ways to get height $1/8$, 
  though, and
  we must exclude them by showing that they are inconsistent with $2$-torsion:
  in particular, by proving that the height of such a section would not be the
  same as that of its $2$-torsion translates.
  Note that there are two $2$-torsion sections that pass through component $4$
  of the $A_7$ fibre and the nonzero component of four of the $A_1$'s and one
  that meets the zero component of the $A_7$ and the nonzero components of
  all $A_1$'s.

  A more
  careful examination of \cite{cox-zucker} shows that there are three
  possibilities for $s$:
  \begin{enumerate}
  \item The section $s$ meets the zero section once, the nonzero components of
    all $A_1$ fibres, and a component of $A_7$ three away from the origin
    (so that its height is $4+2-8\cdot (1/2) - 3 \cdot 5/8 = 1/8$).
    Translating this
    by the $2$-torsion section that meets the nonzero components of all $A_1$
    gives a section of height $4 + 2e - 3 \cdot 5/8$, where $e$ is the intersection
    with the zero section.  Since $e$ cannot be $-1$ this is not $1/8$.
  \item The section $s$ meets a component of $A_7$ three away from the origin and
    four nonzero components of $A_1$ fibres.  If we translate by a torsion point
    passing through the far component of the $A_7$, the height becomes
    $4 + 2e - 7/8 - n/2$, where $n$ is the number of nonzero components of
    $A_1$ fibres now met, so we must have $e = 0, n = 6$.  However, this holds
    for both such torsion points, which is impossible because their difference
    passes through the nonzero component of every $A_1$.
  \item The section $s$ meets a component of $A_7$ one away from the zero
    component and the nonzero components of six $A_1$ fibres.  This is excluded
    by the same argument as in the first case.
  \end{enumerate}
\end{proof}

\begin{cor}\label{cor:disc-is-64}
  The discriminant of the Picard lattice of $S_5'$ is $-64$.
\end{cor}

\begin{proof} Since we have found a sublattice of maximal rank and discriminant
  $-64$, the discriminant of the full lattice is $-64/n^2$ for some integer
  $n$.  One of the $14$ fibrations defined from $3$-planes through
  $4$ singular points as described 
  above does indeed have reducible fibres as in
  Proposition~\ref{prop:disc-more-than-16}, which shows that $n^2 = 1$.
\end{proof}

Listing $ADE$ configurations among the $50$ known curves, we look for those that
have multisection degree $2$ and have fibre class $F$ such that enlarging the
lattice by $F/2$ produces a lattice isomorphic to that of the Kummer surface of
$E \times E'$ (which can be described as $D_8 \oplus D_8 \oplus U$, where
$U$ is the hyperbolic lattice $\langle x,y\rangle$ with
$x^2 = y^2 = 0, x \cdot y = 1$, or in various other ways).

\begin{prop} There is a fibration on $S_5'$ with multisection degree $2$,
  two reducible $I_8$ fibres, and fibres of degree $8$.
\end{prop}

\begin{proof} This is easily calculated from knowledge of the Picard
  lattice and the curves we have found on $S_5'$.
  One way to find its equations is to write
down the two reducible fibres $F_1, F_2$ as subschemes of $S_5'$.  Let $Q_1$
be a quadratic polynomial vanishing on $F_1$ but not all of $S_5'$, so that
$S_5' \cap (Q_1 = 0) = F_1 \cup R$.  Then $F_2 \cup R$ is also cut out by a
quadratic polynomial $Q_2$, and the map given by $(Q_1:Q_2)$ is the desired
genus-$1$ fibration, though with base locus $R$.
\end{proof}

Projecting away from a suitable singularity, we express the generic fibre of
this
fibration as a curve of degree $6$ in $\P^4$, and from there by projecting away
from a divisor of degree $2$ we get a singular quartic.  In contrast to the
slightly more complicated situation of Proposition~\ref{prop:d6-d4s-a1-fib},
this allows us to find a smooth model in $\P^3$ by mapping by the linear
system of quadrics through the singular points.  We can then find the
Jacobian of this fibration (call it $J$).

\begin{prop}\label{prop:eqn-j-17-8}
  $J$ is isomorphic to the elliptic curve defined by
  $y^2 = x^3 + (t^4+13t^2+16)x^2+4t^4x$.
\end{prop}

The curve
thus defined spreads out to a K3 surface with the Picard lattice of a Kummer
surface of the product of two elliptic curves.
Using the calculations of \cite{shtukas}, we compare this fibration to the
fibrations on the Jacobian of a product of two elliptic curves with full
level-$2$ structure, and we indeed find a match with $p =(-5a-4)/4$, where
$a^2 - a - 4 = 0$ is a generator of the field $\Q(\sqrt{17})$.  This produces
one of the curves of conductor $\p_2^3$, where $\p_2$ is one of the primes
above $2$ in $\Q(a)$.  Let $q = p^\sigma$ where $\sigma$ is the nontrivial
automorphism of $\Q(a)$.

Starting again from the other side, we take the elliptic curves
$E: y^2 = x(x-1)(x-p), E^\sigma: y^2 = x(x-1)(x-q)$ and consider the Kummer
surface of their product.  To construct these curves from the given ones whose
conductor has norm $8$, we need to take a quadratic twist by an element whose
norm is the negative of a square, so the Kummer surface we are actually
interested in is a quadratic twist of this one by $-1$.  (We use
$E \times E^\sigma$ because the formulas for its fibrations are already
available in a convenient form.)  Since $p$ and $q$ are quadratic conjugates,
it is easy to find a model of this surface that is defined over $\Q$ and to
pull back the fibrations with two $I_8$ fibres
to it.  Only one of them has bad fibres in locations that match those of $J$.
It has fibres that are defined over $\Q$, so it is a simple matter to rewrite
this fibration in terms of equations defined over $\Q$.
As in Section~\ref{ex:d3-rt13}, the surfaces
are isomorphic over $\Q$ up to a possible twist by $17$ and we need to choose
the correct $\Q$-model of the Kummer surface to obtain an isomorphism over $\Q$.

\subsection{Level $(2)$ over $\Q(\sqrt{17})$}\label{ex:d2-rt17}
In this section we consider a $\Q$-curve, for which the conjecture is
already known.  Nevertheless there are some interesting questions to study.

Since $2$ splits in $\Q(\sqrt{17})$, there are $4$ cusps and hence $4$
Eisenstein series in weight~$2$.  It turns out that the ring of modular forms
is generated by these together with a cusp form of weight~$2$ and $3$
cusp forms of weight~$4$.

\begin{prop}\label{prop:d8-n17-todeg6}
  The quotient of the Hilbert modular surface $H_{(2),\Q(\sqrt{17})}$
  by the Atkin-Lehner involutions is birational to a K3 surface of degree
  $6$ in $\P^4$ with singularities of type $A_1, A_2, A_3$.
\end{prop}

\begin{proof}
Since the cusp form is fixed by both Atkin-Lehner
involutions, we immediately pass to the quotient by these, obtaining a surface
in $\P^5$.  This surface is singular along a line, away from which we project
to create a surface in $\P^3$ singular along a conic.  The quadratic forms
vanishing on this conic give a projectively normal embedding of a K3 surface
into $\P^4$ with the usual degree $6$.  The singularities are easily
checked in Magma.
\end{proof}

Projecting away from the unique $A_3$ gives a new model $S_3 \subset \P^3$
with a rational $A_2$ singularity, two rational $A_1$s, and four
pairs of $A_1$ singularities defined and conjugate over $\Q(\sqrt{2})$.
The outer components of the $A_3$ are now lines over $\Q(\sqrt{17})$.
Using this information we are able to determine the entire Picard group.
It has rank $19$; this is as expected,
because the elliptic curves over $\Q(\sqrt{17})$ of conductor
$(2)$ are $\Q$-curves.  The determinant of the Gram matrix is $192$
and the generators of the Picard group are all defined over
$\Q(\sqrt{2},\sqrt{17})$.

Since some of the curves of conductor $(2)$ are $2$-isogenous to their Galois
conjugates, there should be a correspondence with the Kummer surface of
a product $E \times E'$, where $E, E'$ are $2$-isogenous.  Such a surface
has Picard discriminant $32$.  Other curves in the isogeny class are
$8$-isogenous to their conjugates, resulting in a more complicated
Kummer surface.  Our considerations in this paper are only up to maps of
finite degree, so we need not treat these.

Let us first describe how we found a correspondence between the Hilbert
modular surface and the Kummer surface of $E \times E'$ defined over
$\Q(\sqrt{17})$.  As above, we construct a model $S_3$
of the Atkin-Lehner quotient
of $H_{(2),\Q(\sqrt{17})}$ as a K3 surface in $\P^3$.  By standard calculations
we find:

\begin{prop}\label{prop:d6d5a3a1a1}
  There is a genus-$1$ fibration on $S_3$ with $\bar \Q$-multisection
  degree $2$ and $\Q$-multisection degree $4$.  The reducible fibre
  types are $I_2^*, I_1^*, I_4, I_2, I_2$, and a Weierstrass equation for
  the Jacobian is
  \begin{align*}
  y^2 &= x^3 + (-2t^3 - 359/16t^2 - 99/2t + 33)x^2 + \\
  &\qquad (t^6 + 55/2t^5 + 871/4t^4 + 4937/8t^3 + 523/2t^2 - 842t + 264)x.\\
  \end{align*}
  The Jacobian admits a model
  $S_7 \subset \P^7$ with $8$ $A_1$ singularities, all defined over $\Q$,
  and no others;
  its Picard group has discriminant $48$ and the generators are all defined
  over $\Q(\sqrt{17})$.  It has a fibration with bad fibre types
  $I_2^*, I_2^*, I_0^*$ and a $2$-torsion section.
\end{prop}

The importance of this fibration is that
vertical curves together with the sections of order $1$ and $2$ contain a
configuration that is a fibre of a fibration with reducible fibres
$I_2^*, {I_0^*}^2, I_2^2$ and full level-$2$ structure.
One of the $2$-torsion
sections must pass through the reduced component of the $D_6$ that is on
the same side as the zero section; the formulas of \cite{cox-zucker}
imply that, in order to have height $0$, it must
pass through the nonzero components of the $A_1$ fibres.  Accordingly the
quotient by this $2$-isogeny will have reducible fibres $D_8, D_4^2$.
Every K3 surface with a $D_8, D_4^2$ fibration is the Kummer surface of a
product of elliptic curves.

Unfortunately not all of the components of the fibres of the fibration on
$S_7$ are rational.
Although we believe that there should be a correspondence over $\Q$, we have
not been able to exhibit it, for reasons to be discussed in
Remark~\ref{rem:why-not-over-q}.
We therefore work over $\Q(\sqrt{17})$ for the rest of this section.

Having constructed the ${I_2^*}^2, I_0^*$ fibration as an
elliptic curve over $\Q(\sqrt{17})(u)$, we use a model in $\P^7$ to help
construct the desired $I_2, {I_0^*}^2, I_2^2$ fibration.  Magma immediately
produces the $2$-isogenous elliptic surface with $I_4^*, {I_0^*}^2$ fibres and
a point of order $2$.

Since the Picard discriminant was $48$ before the last $2$-isogeny, the
Mordell-Weil generator must have been of height $3$.  It cannot be of height
$3/2$ on the new surface, since by \cite{cox-zucker} every section on a
surface whose reducible fibres are all of type $I_{2n}^*$ must have integral
height; thus the height must be $6$ and we expect that the surface thus
obtained will be the Kummer surface of a product of $6$-isogenous elliptic
curves.

\begin{thm} The elliptic surface we have constructed with $I_4^*, {I_0^*}^2$
  fibres is isomorphic to the Kummer surface of
\elllmfdb{2.2.17.1}{4.1}{a/8} $\times$ \elllmfdb{2.2.17.1}{4.1}{a/10}.
\end{thm}

\begin{proof} We write down a suitable fibration on the product using
  the formulas from \cite{shtukas}.
  That is, by writing down the $D_8, D_4^2$ fibrations on the Kummer surface of
  a general $E \times E'$ with full level-$2$ structure we see how to choose
  the parameters in the Legendre family so that the local invariants match
  and then we verify that the fibrations are exactly isomorphic.
\end{proof}

Strictly speaking, the Oda-Hamahata
conjecture predicts a correspondence with the
Kummer surface of $E \times E^\sigma$, but an isogeny from $E'$ to $E^\sigma$
induces one from $\Kum(E \times E')$ to $\Kum(E \times E^\sigma)$.

\begin{remark}\label{rem:why-not-over-q}
  In attempting to construct a correspondence over $\Q$,
  the basic problem is that the curves in isogeny
  class \elllmfdb{2.2.17.1}{4.1}{a} are $2$- or $8$-isogenous to their Galois
  conjugates and so the Picard discriminants of the obvious Kummer surfaces
  to consider over $\Q$ are not multiples of $3$.  Starting from the surface
  of Picard discriminant $192$, therefore, we must at some point find a
  $3$-isogeny to make the $3$-adic valuation of the discriminant even.
  While this is not difficult to do, it always resulted in a surface whose
  Picard discriminant was divisible by $9$.  To pass to the desired Kummer
  surface would then require taking a $3$-Jacobian of a fibration with
  no section.  Again, it is easy to find such a fibration, but although
  many surfaces were investigated in the hope that one of them would have
  such a fibration defined over $\Q$, none had this property.  Thus
  the desired correspondence could not be constructed over $\Q$.
\end{remark}

\subsection{Level $(3)$ over $\Q(\sqrt{2})$}\label{ex:d3-rt2}
This example is somewhat similar to the previous one.  The elliptic curves
are in the isogeny class \elllmfdb{2.2.8.1}{9.1}{a}; they are curves
defined over $\Q$ twisted by $2 + \sqrt 2$.  Since $N_{\Q(\sqrt 2)/\Q}(2+\sqrt 2)$
is a square in $\Q(\sqrt 2)$, this means that they are isomorphic to their
Galois conjugates.  The curves do not have complex multiplication, so the
Kummer surface of $E \times E^\sigma$ will have geometric Picard number $19$.

The ring of modular
forms is generated by forms of weight $2, 2, 2, 4, 6, 6, 8$; there are
two Eisenstein series of weight~$2$, and all the rest of the generators may
be taken to be cusp forms.

\begin{lemma}\label{lem:embed-p9}
  There is a birational map from the Atkin-Lehner quotient to a K3
  surface in $\P^9$.  This surface admits a genus-$1$ fibration
  whose Jacobian has Weierstrass equation
$$y^2 = x^3 + (2t^4 + 3t^3 - 7t^2 + 6t + 14)x^2 + (18t^5 + 45t^4 - 36t^3 - 90t^2 + 18t + 45)x$$
  with reducible fibre types $I_0^*, I_6, I_4, I_2^4$, full level-$2$
  structure, and a section of height $5/12$ defined over $\Q(\sqrt{2})$
  whose $x$-coordinate is $6t^2-6$.  The discriminant of the Picard lattice
  is $40$.
\end{lemma}

\begin{proof}
  Since $3$ is inert in $\Q(\sqrt{2})$, there is a unique Atkin-Lehner
  operator.   We consider the map given by the sections of $\O(4)$ in its
  $+1$ eigenspace; these give a map to $\P^{12}$ whose image we compute
  by Algorithm~\ref{alg:find-image}  to be a surface of degree $24$.
  The projection away from the image of the cusps is of degree $22$ and
  is singular along a line; projecting away from
  this line produces a surface of degree $16$ in $\P^9$ defined by quadrics.
  Using Magma we check that this surface has trivial canonical sheaf
  and then that it is a K3 surface.
  On this surface we are fortunate to be able to find some curves of genus $1$
  by setting one of the
  coordinates to be $0$.  That is, some of the components of the subscheme
  of $S_9$ defined by $x_i = 0$ turn out to be curves of arithmetic genus $1$
  and degree $4$.  We find the fibrations of which these curves are fibres
  and pass to the Jacobian;
  for one of them, the reducible fibre types are as described and as in
  Remark~\ref{rem:reduce-weierstrass} we find the given Weierstrass
  equation.  The section may be found by standard techniques, although
  in this example it is easy enough to spot it by reducing mod small $p$
  and searching.

  The calculation of the discriminant follows from the determinant
  formula \cite[Corollary 6.39]{schutt-shioda} as before: it is
  $4 \cdot 6 \cdot 4 \cdot 2^4 \cdot 5/12 \cdot (1/4)^2 = 40$.
\end{proof}

\begin{remark}\label{rem:need-5-to-work-over-q}
  Note that, as in the previous example, there is a prime
  $p>2$ dividing the Picard discriminant which is the isogeny
  degree between some curves in this isogeny class but does not divide the
  isogeny degree between a curve and its Galois conjugate.  Thus, in order
  to find the desired correspondence over $\Q$, we would at some point need
  to use a map of degree divisible by $5$.
  In fact the Picard lattice is isomorphic to that
  of a K3 surface with $4$ fibres of type $I_5$, a section of height
  $8/5 = 4 - 2(6/5)$, and a $5$-torsion section.  The $5$-isogeny goes to
  a surface of Picard discriminant $200$; however, as in Remark
  \ref{rem:why-not-over-q} it seems
  impossible to find a genus-$1$ fibration on this surface
  defined over $\Q$ with multisection degree $5$, so we cannot find the
  correspondence with a Kummer surface over $\Q$.
\end{remark}

We proceed to find a correspondence with a Kummer surface over
$\Q(\sqrt 2, \sqrt 3)$, the $2$-division field of two of the curves in
the isogeny class.

\begin{prop}\label{prop:model-p3}
  The elliptic surface of Lemma \ref{lem:embed-p9} admits a model
  in $\P^3$ with three rational $A_2$
  singularities and one pair each of conjugate $A_2$ and $A_1$.
  It has a fibration with reducible fibre types $I_2^*, {I_0^*}^2, I_2^2$
  and full level-$2$ structure.  One of the curves $2$-isogenous to this
  one has reducible fibres $I_4^*, {I_0^*}^2$ and a point of order $2$.
  The equation of the generic fibre of this fibration is
  $$y^2 = x^3 + (-10t^3 - 88t^2 + 98t)x^2 + (t^6 - 
  4t^5 + 6t^4 - 4t^3 + t^2)x.$$
\end{prop}

\begin{proof} As before, this is a standard calculation for which the
  reader is referred to \cite{code}.
\end{proof}

\begin{remark} For the reason to seek out this type of fibration,
  compare Section~\ref{ex:d3-rt13}; the point is that the isogenous
  fibration is found on the Kummer surface of a product.
\end{remark}

As before, the formulas from
\cite{shtukas} identify this fibration over $\Q(\sqrt 2, \sqrt 3)$ with one
on the Kummer surface of $y_0^2 - x_0(x_0-1)(x_0-p) = y_1^2 - x_1(x_1-1)(x_1-q)$
where $p = 6+2\sqrt{6}, q = -485+198 \sqrt{6}$.  These are twists of
the curves \elllmfdb{2.2.8.1}{9.1}{a/3}, \elllmfdb{2.2.8.1}{9.1}{a/4} by
the same number $-2+\sqrt{2}-\sqrt{3}+\sqrt{6}$, so the Kummer surfaces of
the two products are isomorphic over $\Q(\sqrt{2},\sqrt{3})$.  We
thus conclude:

\begin{cor}\label{corr:done-d8-n9}
  Over $\Q(\sqrt{2},\sqrt{3})$ there is a correspondence between the
  Hilbert modular surface $H_{(3),\Q(\sqrt{2})}$ and the product of an
  elliptic curve over $\Q(\sqrt{2})$ of conductor $(3)$ and its conjugate.
\end{cor}

\begin{proof}
  Above we constructed a dominant map from $H_{(3),\Q(\sqrt{2})}$ to
  one such Kummer surface.  As before, any two such Kummer surfaces admit
  maps induced by isogenies of the corresponding abelian varieties.
\end{proof}

\subsection{Level $\p_{17}$ over $\Q(\sqrt{2})$}\label{ex:d17-rt2}
In this case, where again the conjecture was not previously known, we
are able to give a correspondence over $\Q(\zeta_8)$.  In this section
we let $K = \Q(\sqrt{2})$ and $I = (5-2\sqrt{2})$, a prime ideal of
$\O_K$ of norm $17$.  The Hilbert modular variety $H_{I,K}$
is a subvariety of $\P(1^3,2^4)$.

\begin{prop}\label{aq-d8-n17}
  The image of the Atkin-Lehner quotient of $H_{I,K}$ by $\O(2)$ is
  a surface in $\P^{5}$ of degree $12$ defined by equations of degree
  $2, 3, 3, 4, 4, 4, 4, 4, 4$ and is birational to a surface with a
  genus-$1$ fibration whose Jacobian is a K3 surface.  It has two conjugate
  $I_0^*$ fibres defined over $\Q(\sqrt{17})$ as well as one $I_4$ and
  three $I_2$ fibres defined over $\Q$.
  There are two sections defined over $\Q(\sqrt{2})$ of height $3/4$ and $3/2$
  and height pairing $0$.  The discriminant of the Picard lattice is $-144$.
\end{prop}

\begin{proof} The first statement is a routine calculation.  The image
  of the cusps is a singularity $C$ such that the projection away from $C$ has
  degree $8$ in $\P^4$.  Considering forms of degree $2$ vanishing along
  the line where this surface is singular and the worst of the singular points,
  we map to a surface $S_9$ of degree $21$ in $\P^9$.  Now $S_9$ is singular
  along two lines; projecting away from their span, we find a surface of
  degree $10$ in $S_5$ defined by polynomials of degree $2, 2, 3, 3$.
  Although this is still singular on a line, we can find a genus-$1$ fibration
  on it (defined by linear forms and with base scheme consisting of
  the singular line doubled and a certain singular curve of degree $4$).
  The Jacobian of this fibration is an elliptic
  surface with Euler characteristic $24$, which is therefore a K3 surface.
  Once the equation is found, the reducible fibres are easy to check
  and the sections can be found be standard techniques.

  Given this information, we verify that the subgroup of the Mordell-Weil
  group that we have found is saturated at $2$ and $3$.  Since the
  torsion subgroup has order $2$, it follows that the Picard discriminant
  is $-(4^2 \cdot 4 \cdot 2^3 \cdot 3/2 \cdot 3/4)/2^2 = -144$.
\end{proof}

\begin{remark} The minimal desingularization of $S_5$ has the same Hodge
  diamond as a K3 surface, but it is not itself a K3 surface.  This
  follows from Kodaira's formula for the canonical bundle
  of a surface with a genus-$1$ fibration, combined with the existence
  of a point of $\P^1$ where the fibre is nonreduced on $S_5$ but is
  a smooth rational curve on the Jacobian.  The situation is analogous to
  that of an Enriques surface, for which the Hodge diamond is the same as
  that of a rational elliptic surface.

  A somewhat similar calculation 
  exhibits a genus $1$ fibration on $H_{I,K}$; alternatively one can verify
  that the ramification locus of the quotient map is disjoint from the
  generic fibre of the fibration.  As mentioned in
  Remark~\ref{rem:hamahata-error}, this shows that Hamahata's claim that
  $\kappa(H_{I,K}) = 2$ is incorrect.
\end{remark}

Having determined a K3 surface quotient of the Hilbert modular surface, we
must now improve the model.  As usual, we do this by embedding in $\P^5$
and successively projecting away from noncanonical singularities and
applying the Veronese embedding until the desired model is reached.
In this case the procedure is quite lengthy, requiring $10$ projections,
but once this is done we can project away from the unique $A_3$ point to
find a model $S_7 \subset \P^7$ whose only singularities are $11$ ordinary
double points.  The fibres of the fibration are of degree $3$.  In this model,
the reducible fibres are two of type $I_0^*$, defined over $\Q(\sqrt{17})$ and
given in this model by a double
line, a line, and three rational nodes; one $I_4$, consisting of a rational
line, two conjugate lines over $\Q(\sqrt{2})$, and a rational node; and
three $I_2$, whose components are a node and a cubic singular there.

We might hope
to reach a Kummer surface of a product by taking the Jacobian of a
fibration with a $3$-section but no rational section.  There are two
problems with this: first, no rational Picard class that is not a multiple
of $3$ has intersection
divisible by $3$ with all rational Picard classes, and second, the
lattice containing the Picard lattice with index $3$ is not the Picard lattice
of the Kummer surface of a product of two elliptic curves but rather
another lattice of the same rank and discriminant.

\begin{prop}\label{prop:find-2-section}
  There is a fibration on $S_7$ with a $2$-section but no section.
  It has reducible fibres $I_4^*, I_6, I_2, I_2$,
  torsion of order $2$, and a section of height $3/2$.  The Jacobian has
  Weierstrass model
  $$y^2 = x^3 + (17t^3 - 14t^2 + 7t - 2)x^2 + (4t^4 -
  13t^3 + 15t^2 - 7t + 1)x,$$
  and the $x$-coordinate of a section is $1$.
\end{prop}

We proceed as before
to find a good projective model of this surface.  In this case it is easy
to find a K3 surface model, but its singularities are undesirable and we
have to do some work to find a model $S_6 \subset \P^6$ with singularities
of type $A_3, A_3, A_2, A_1$.  It turns out that the full Picard lattice is
defined over $\Q(\sqrt{2})$; we no longer need to adjoin $\sqrt{17}$.

Unfortunately there is still no rational divisor class that can be divided by
$3$; this remains true for other isogenous surfaces, so we consider a
genus-$1$ fibration defined over $\Q(\sqrt{2})$ with multisection degree $3$.

\begin{prop} There is a fibration on $S_6$ defined over $\Q(\sqrt{2})$ with 
two $I_4^*$ fibres, torsion of order $2$, and multisection $3$.
\end{prop}

Again, the reason for wanting such a fibration is that the
$2$-isogenous surface has two $I_2^*$ fibres, four of type $I_2$, and
$2$-torsion, and this is one of the type of fibrations on the Kummer
surface of $E \times E'$.

The rest is now routine; as in the previous examples, we use the results of
\cite{shtukas} to match this fibration with one on the Kummer surface of
$E \times E'$.  It is necessary to extend the field from $\Q(\sqrt{2})$ to
$\Q(\zeta_8)$ where the $2$-torsion of $E, E'$ is defined, but in the
end we prove the following result:

\begin{thm}\label{thm:d8-n17}
  The Hilbert modular surface $H_{I,K}$ admits a rational map of finite degree
  defined over $\Q(\zeta_8)$
  to the Kummer surface of $E \times E'$, where $E, E'$ are elliptic curves
  over $K$ of conductor $I$.
\end{thm}

As before, if this is true for one choice of $E, E'$ it is true for all
of them.

\subsection{Level $\p_{31}$ over $\Q(\sqrt{5})$}\label{ex:d5-p31}
The Kodaira dimension of the Hilbert modular surface corresponding to the
ideal of smallest norm in $\O_{\Q(\sqrt{5})}$ giving a surface with $p_g = 1$
was left undecided by Hamahata.  Here we will prove that its Kodaira dimension
is $1$ and verify the Oda-Hamahata
conjecture.  We begin as usual by computing the
ring of Hilbert modular forms, finding it to be generated by $3, 6, 4$
forms of weight $2, 6, 10$ respectively.  We also determine the Atkin-Lehner
involution and verify that it has $8$ isolated fixed points, none of which is
singular on the surface.

\begin{prop}\label{prop:d5-n31-k3}
  The Atkin-Lehner quotient of the Hilbert modular surface
  $H_{(6-\sqrt{5}),\Q(\sqrt{5})}$ admits a fibration by curves of genus $1$ whose
  Jacobian is a K3 surface.
\end{prop}

\begin{proof}
  We begin by studying the quotient.  There are too many modular forms of
  weight~$10$, so we use the forms of weight~$6$ invariant under Atkin-Lehner
  to map to a surface of degree $18$ in $\P^8$.
  By projecting successively from the
  cusp, the new singular point in its image, and the singular line we reach a
  surface of degree $9$ in $\P^4$.
  Although this surface does not have canonical
  singularities, it is sufficiently well-behaved that Magma can compute a
  canonical divisor map whose codomain is $\P^1$ and whose general fibre is of
  genus $1$, verifying that the surface is not of general type.
\end{proof}

The surface
has one $I_3$, ten $I_2$, and one $I_1$ fibre, indicating that its Jacobian
is a K3 surface.  However, we claim that its Kodaira dimension is $1$.

\begin{prop}\label{prop:d5-n31-kd1}
  The quotient of $H_{\p_{31},\Q(\sqrt{5})}$ by the Atkin-Lehner
  involution has Kodaira dimension $1$.
\end{prop}

\begin{proof} Since the Euler characteristic is $24$, this is equivalent to
  the statement that the fibration has at least one multiple fibre.
  This is easy to verify computationally
\end{proof}

\begin{cor}\label{cor:d5-n31-kd1}
  The Hilbert modular surface $H_{\p_{31},\Q(\sqrt{5})}$ has Kodaira dimension $1$.
\end{cor}

\begin{proof} Let $F$ be the fixed locus of the Atkin-Lehner involution.
  Since $F$ consists of isolated nonsingular points, its image on the
  Atkin-Lehner quotient consists of ordinary double points.  Thus,
  on the minimal resolution of the quotient, the double cover is ramified
  precisely at the $8$ exceptional curves of self-intersection $-2$ above
  these points.  Since the fibration on the quotient is well-defined at these
  points, the curves are vertical for the fibration, and hence the
  general fibre of the genus-$1$ fibration on the quotient is disjoint from
  the ramification locus.  It follows that the general fibre pulls back to
  a curve of genus $1$.  Thus $H_{\p_{31},\Q(\sqrt{5})}$ admits a genus-$1$
  fibration and cannot be of general type.
\end{proof}

We now proceed to exhibit a correspondence between the Jacobian of the
Atkin-Lehner quotient and the Kummer surface of the product of an elliptic
curve in the isogeny class and its conjugate.  This is a relatively easy
example because the diameter (Definition~\ref{def:diameter}) of the
isogeny class is $4$ (cf.~Claim~\ref{claim:disc-pic})
and there are curves in the class with full level-$2$
structure.  As usual, we do most of the work on the side of the quotient
rather than that of the Kummer surface.

The eight ordinary double points mentioned above are components of $I_2$
fibres and are defined over a number field whose Galois group has order $64$.
It would be unpleasant to work with a K3 surface over such a large number
field, so  we begin by passing to the
$2$-isogenous elliptic surface $\E$ on which the images of these fibres are of
type $I_1$.  The reducible fibres of $\E$ are of type $I_6, I_4, I_4, I_2$.

\begin{prop}\label{prop:rank-4}
  The Weierstrass model of $E$ can be taken to be
  \begin{align*}
    y^2 &= x^3 + (-43t^4 + 86t^3 + 11t^2 - 38t + 5)x^2\\
    &+ (496t^8 - 1984t^7 + 1488t^6 + 1984t^5 - 1488t^4 - 992t^3)x.
  \end{align*}
  There are four independent sections giving a
  discriminant of $4/3$ relative to the height pairing.
\end{prop}

\begin{proof}
  The reader is referred to \cite{code}.
\end{proof}

Since the torsion subgroup has order $2$, this implies a Picard discriminant of
$6 \cdot 4 \cdot 4 \cdot 2 \cdot 4/3 \cdot 1/4 = 64$ up to sign.
The generators are all defined over $\Q(\sqrt{5})$.
In a few hours we can enumerate the $191$ definite lattices in the genus
that are the frame lattices for elliptic fibrations on the surface.

We are able to exhibit a model of this surface in $\P^5$ with $6$ ordinary
double points and no more.  There do exist fibrations defined over $\Q$
with no section, but for historical reasons we use one that merely has
no $\Q$-rational section.

\begin{prop}\label{prop:unnecessary-fib}
  There is a fibration with no $\Q$-rational section with  reducible fibre
  types $I_1^*, I_0^*, I_2^3$ and is such that in the $\P^5$ model
   the $I_0^*$ fibre contains a conic. 
\end{prop}

The equation is a bit too complicated to show here, but
we easily find the general fibre, pass to the Jacobian, and write down some
sections.  With some effort we find a model of this new K3 surface as
$S_6 \subset \P^6$ with one $A_2$ and three $A_1$ singularities.

\begin{prop}\label{prop:fib-s6}
  The dimension of the Galois-fixed subgroup of the Picard group of $S_6$ is
  $14$, and $S_6$ has a fibration defined over $\Q$
  with multisection degree $2$ over $\Q$ and over $\bar \Q$.  The reducible
  fibres of the Jacobian are of type $I_4^*, I_0^*, I_4$, and the torsion
  subgroup has order $2$.  Its Weierstrass equation is
  $$y^2 = x^3 + (1/4t^3 + 65/4t^2 + 1724t)x^2 + 
  (31t^4 - 4960t^3 + 198400t^2)x.$$
  There is a section of height $1$.
  \end{prop}

The fibres of this fibration have degree $8$; the corresponding map to $\P^1$
and the Jacobian of the general fibre can be found without too much difficulty.

Thus we have a surface of Picard discriminant $-16$; however, it is not
the same lattice as that of the Kummer surface of a product.
It is difficult to find a good model of this
surface; we make do with $S_8 \subset \P^8$ whose singularities are of type
$A_4, A_3, A_1, A_1, A_1$.  Although the full Picard group is still defined
only over $\Q(\sqrt{5})$, the subgroup defined over $\Q$
now has rank $16$.

\begin{prop} There is a fibration on $S_8$ defined over $\Q$
  whose Jacobian has fibres of
  types $III^*, III^*, I_2, I_2$ and that has $2$-torsion.  The Weierstrass
  equation is
  $$y^2 = x^3 + 2180x^2 + \frac{-16/5t^2 + 167872t - 19681280}{t}x.$$
\end{prop}

The $III^*$ fibres are rational, so the existence of a section forces
all of their components to be rational.  The $I_2$ fibres, however, are
still conjugate over $\Q(\sqrt{5})$, so the rational Picard rank is $17$.
(Note that the geometric Mordell-Weil rank is $0$ for this fibration.)

Before comparing to the Kummer surface, we pass to yet another K3 surface,
this time with Picard discriminant $1$.

\begin{prop}\label{prop:two-e8} On $S_8$ there is a fibration 
with two $II^*$ fibres and trivial Mordell-Weil group.
The Jacobian is isomorphic to
$$C_{E_8}: y^2 = x^3 + 545x^2 + 74272x + (1/25t^2 + 69760t + 98406400)/t.$$
\end{prop}

\begin{proof}
  Finding the Picard
  class $C$ of such a fibration and expressing it in terms of the sections and
  vertical curves of the previous fibration is not very difficult.
  However, using our methods as described above it would be painful to
  find equations for the fibration; the small Picard discriminant makes it
  difficult or impossible
  to find a model in a projective space of low dimension with
  singularities of small degree.  Instead we proceed as follows.  Spreading
  out the equation above to a surface in $\P^2 \times \P^1$ and placing it
  into $\P^5$ by the Segre embedding, we find a model with two $A_6$ and
  two $A_1$ singularities.
  We write down the two reducible fibres as subschemes
  $F_1, F_2$
  of this surface.  Since $2H-C$ is effective, the fibration can be
  defined by the equations $(p_1:p_2)$, where $p_i$ is a polynomial of degree
  $2$ vanishing on $F_i$.  This constrains each of the $p_i$ to lie in a
  $3$-dimensional projective space.  We now reduce mod $7$ and test
  possibilities until we find one whose base scheme has the correct degree,
  then verify that it is in fact an elliptic fibration with two $II^*$ fibres.
  It turned out to be easy to recognize the $p_i$ as reductions of polynomials
  from characteristic $0$ and thus to find the equations of the fibration.
\end{proof}

The following proposition facilitates the comparison with the Kummer surface.
\begin{prop}\label{prop:d16}
  The surface with generic fibre $C_{E_8}$ also admits a fibration with
  an $I_{12}^*$ fibre whose Weierstrass equation is
  $$y^2 = x^3 + (8t^3 - 271t^2 - 32t + 1022)x^2 + 961x.$$
\end{prop}

\begin{proof} This can be proved by a trick similar to, but simpler than,
that of Proposition~\ref{prop:two-e8}.
Letting $x_0,x_1,x_2$ and $y_0,y_1$ be the coordinates on $\P^2 \times \P^1$,
we find that $(x_0:x_2)$ is also a genus-$1$ fibration on the spreading-out
of $C_{E_8}$ to a surface in $\P^2 \times \P^1$.  (This was found from the
Segre embedding by considering the linear system $H - 2E$, where $E$ is
the exceptional divisor at one of the $A_8$ singularities of the embedding.)
It is now routine to find the equation, and by translating the $2$-torsion
point to the origin and scaling away some obvious powers of $2$ we reach
the form indicated above.
\end{proof}

The fact that this fibration has an $I_{12}^*$ fibre means that over $\C$
it is connected with Kummer surfaces by means of a $2$-isogeny known as a
Shioda-Inose structure \cite[Section 6]{morrison-large}.

For the rest of this section let $\alpha = \frac{1+\sqrt{5}}{2}$.
We will verify that the fibration is isomorphic to one on a K3 surface
that is dominated by the Kummer surface $\Kum$ of $E \times E^\sigma$, where
$E: y^2 + xy + \alpha y = x^3 + (\alpha+1) x^2 + (\alpha-5) x + (3\alpha-5)$
is an elliptic curve of conductor $\p_{31}$
(namely \elllmfdb{2.2.5.1}{31.1}{a/5}).  Indeed, this curve is isomorphic to
$y^2 = x(x-1)(x-\frac{8\alpha+3}{16})$, so we may apply our general formulas
to find a model for the Kummer surface over $\Q$, taking care as usual to
use the correct twist.  There is a fibration with two
$III^*$ fibres and no section, but since the $III^*$ fibres are not
defined over $\Q$ this model cannot be directly compared with the one found
as an image of the Hilbert modular surface.

\begin{prop} There is a fibration on $\Kum$ with an $I_6^*$ fibre
  and six of type $I_2$.  It has sections defined over $\Q(\sqrt{5})$ but
  none defined over $\Q$.
\end{prop}

Using a model of the Kummer
surface in $\P^4$ with an $A_3$ and nine $A_1$ singularities, we are able
to write an equation for the general fibre, which has degree $12$ in $\P^4$.
The only difficulty is that projection from the worst singularity is a map
of degree $2$; however, it is manageable to find the divisor map given by
the degree-$4$ place supported there.  This gives a model in the standard
form of an intersection of two quadrics in $\P^3$.  From this we wish to
construct an elliptic surface with an $I_{12}^*$ fibre; that is only a matter
of taking an isogeny by the correct $2$-torsion point.  We obtain the
equation
$$y^2 = x^3 + (8t^3 - 271t^2 - 32t + 1022)x^2 + 961x,$$
the same as in Proposition~\ref{prop:d16}.
Thus, in this case, we have proved the strongest form of the Oda-Hamahata
conjecture.

\section{Surfaces of general type}\label{sec:general-type}
We now proceed to discuss the Hilbert modular surfaces of general type
with $p_g = 1$.  For some of these, the Atkin-Lehner quotient is of
general type and we can do nothing.  However, there are other cases in
which the quotient is of Kodaira dimension $0$ or $1$.  We prove 
a lemma to explain the difference.

\begin{lemma}\label{lem:fixed-point-quotient}
  Let $S$ be a smooth surface and $\phi: S \to S$ an involution whose fixed
  points are isolated and smooth, with $\phi$ acting on the tangent space at
  every fixed point by $-1$.  Let $T = S/\phi$.
  Then $T$ has canonical singularities and $K_S = 2\phi^*(K_T)$.
\end{lemma}

\begin{proof} The singularities of $T$ occur only under the fixed points,
  and are locally isomorphic to $\A^2/\pm 1$, which is an $A_1$ singularity.
  For the statement on canonical divisors, let $S'$ be the blowup of
  $S$ at the fixed points and $T'$ the minimal resolution of $T$.
  Let $E_1, \dots, E_n$ be the exceptional divisors of $S'$.  We then
  have $K_{S'} = K_S + \sum_{i=1}^n E_i$ and $K_{T'} = K_T$.  In addition,
  the map $S' \to T'$ is ramified along the $E_i$, so
  the Riemann-Hurwitz formula gives $K_{S'} = \phi^*(K_{T'}) + \sum_{i=1}^n E_i$.
  The result follows.
\end{proof}

\begin{cor}\label{cor:canonical-self-int}
  With notation as above we have $K_S^2 = 2K_T^2$.
\end{cor}

\begin{proof}
  This follows immediately from the push-pull formula:
  $K_S^2 = 2K_S\phi^*\cdot K_T = 2\phi_*K_S \cdot K_T = 2K_T^2$.
\end{proof}
  
\begin{example} One standard example of this situation arises when $S$ is
  an abelian surface and $\phi$ is the negation map.
\end{example}

A converse to Lemma~\ref{lem:fixed-point-quotient} can hold only
under additional assumptions: if the
fixed points are singular, that does not imply that the quotient is not
of general type.  However, if they are not canonical, we expect that
$2K_T^2 < K_S^2$, which may imply that $K_T^2 \le 0$.  If $K_S^2 = 2$,
this implies that $K_T^2 \le 0$, and if $T$ is minimal that means that
$\kappa(T) < 2$.  We will study such examples in Sections
\ref{ex:d8-2p7}, \ref{ex:d5-4p5}, \ref{ex:d5-p5p11}, proving the
conjecture in the last two cases.

\begin{remark} Some of the surfaces of general type have interesting
  special properties.  As one example, we point out
  that the ring of modular forms for level $3\p_5$ over
  $\Q(\sqrt{5})$ is generated by $5$ forms of weight~$2$ and one of weight~$4$,
  and that it is a module of rank $2$ over the subring $R$ generated by the
  forms of weight~$2$.  As a result the Hilbert modular variety in this case
  admits an additional automorphism (not one of the Atkin-Lehner involutions)
  arising from the double cover of the surface $\Proj(R)$.  Though this is
  not the only example, we expect
  that such additional automorphisms are rare, just as in the situation
  of $X_0(N)$ where the only non-modular automorphism of a curve of general
  type occurs for $N = 37$.  Unfortunately
  the quotient by automorphisms acting as $+1$ on $\HHH^{2,0}$ still
  appears to be of general type.
\end{remark}

The remaining cases occur at level $\p_{17}, \p_{23}$ over $\Q(\sqrt{13})$,
which both give surfaces of general type that are difficult to analyze,
and level $(6)$ over $\Q(\sqrt 5)$, where I have not been able to
obtain a usable form of the ring of modular forms from  \cite{hmf}.

\subsection{Level $2\p_7$ over $\Q(\sqrt 2)$}\label{ex:d8-2p7}
The Kodaira dimension of this surface is left open by Hamahata; however,
the software package determines that it has a model with $K^2 = 4$.
Since $\hh^{2,0} = 1$ for this
surface, it is of general type.  Although we are not able to prove the
Oda-Hamahata conjecture in this case, we can describe some interesting
properties of the surface.

The surface has $6$ cusps corresponding
to the divisors of $2\p_7$.  The Atkin-Lehner involution corresponding to
$\p_7$ acts on the set of cusps without fixed points, so the quotient is
still of general type.  On the other hand, the involution corresponding to
$(2)$ fixes the cusps corresponding to $\sqrt{2}$-adic valuation $1$.

Since the Atkin-Lehner involutions act as $+1$ on the eigenform of
weight~$2$, the quotient still has $\hh^{2,0} = 1$.  We will show that the
quotient admits a fibration by curves of genus $1$.

\begin{remark}\label{rem:four-quadrics}
  Computationally, this example is convenient in some ways, since the ring
  of modular forms is generated by the six Eisenstein series and one cusp form,
  all in weight~$2$.  In fact, the Hilbert modular surface is just a complete
  intersection of four quadrics in $\P^6$, although this misleadingly suggests
  that the canonical divisor would be $\O(1)$.  The cusp singularities are
  not canonical, and in fact the canonical divisors are precisely the sections
  of $\O(1)$ that vanish on them, which constitute a $1$-dimensional space
  (as expected, since this is the dimension of the space of cusp forms).
  Taking the quotient by the Atkin-Lehner involutions, we obtain another
  intersection of four quadrics in $\P^6$, but this one is more singular
  and so its canonical divisor is smaller.
\end{remark}

We will show that the Atkin-Lehner quotient of the Hilbert modular surface
$H_{(2(3-\sqrt 2)),\Q(\sqrt{2})}$ admits
a map to $\P^1$ whose generic fibre is of genus $1$ and that has double
and triple fibres.  Thus the minimum section degree is a multiple of $6$
(in fact exactly $6$), which makes the surface difficult to work with
because Magma does not have a function to compute the Jacobian of a curve
of genus $1$ unless there is a multisection of degree at most $5$.
In order to analyze the surface properly, we therefore begin by 
studying the quotient only by the Atkin-Lehner operator at $2$.
This will be a double cover of the full quotient unramified at the general
fibre, and so
it still has Kodaira dimension $1$.  For this surface, which we call
$\E_3$, the general fibre admits a $3$-section.  It is possible (although
not easy) to determine $\E_3$ as an elliptic surface with an explicit
map to $\P^1$ and isomorphism of its generic fibre to a plane cubic.
We state the result as a theorem.

\begin{thm} $\E_3$ is an elliptic surface with two triple fibres and
  a $3$-section.  Its Jacobian is isomorphic to the rational elliptic surface
  $y^2 = x^3 + (t^2 + 3)x^2 + (2t^2 + t + 1/8)x$.
\end{thm}

\begin{proof}
The main steps involved in this calculation are as follows:

\begin{enumerate} 
\item  Having already written down the Hilbert modular surface
  in $\P^6$, we start by writing down the map to $\P^{12}$ given by the
  sections invariant under the Atkin-Lehner involution at $(2)$ and finding
  its image as a surface of degree $32$ in $\P^{12}$.
\item This surface has four obvious singularities, namely the images of the
  cusps.  Projecting away from these produces a surface in $\P^8$ of degree
  $24$ singular along four lines, which are in bijection with the cusps
  of the Atkin-Lehner quotient.
\item We consider the linear system of quadrics vanishing twice on three of
  these lines and once on the other one (the choice does matter
  computationally; we single out the line corresponding to the cusp $0$).
  This linear system defines a map to a surface of degree $14$ in $\P^6$,
  defined by $3, 8, 2$ equations of degree $2, 3, 4$ respectively.
\item Although this surface has bad singularities, a canonical sheaf
  computation in Magma yields a useful result: in particular, the
  divisor map associated to this sheaf has codomain $\P^1$ and generic fibre
  of degree $6$ and arithmetic genus $1$.
\item There is a curve on the surface that gives a $3$-section of the
  fibration (it happens to be isomorphic to an elliptic curve of conductor
  $14$ with $6$ rational torsion points).  Projecting the generic fibre
  away from the intersection gives a model of degree $3$ in $\P^3$, and
  removing the linear equation gives a plane cubic.
\item Simplifying the model by moving some special fibres to lie above
  small points of $\P^1$ and then using general procedures for minimizing
  and reducing equations of elliptic surfaces gives the model above.  
\end{enumerate}

Having done this computation, we find that there are $3$ values of $t$
for which the plane cubic becomes a triple line; one is rational and the
other two are conjugate over $\Q(\sqrt{-3})$.  The rational one, however,
is reduced in the model in $\P^6$ and is therefore not a multiple fibre of
the surface.
\end{proof}

\begin{remark} The $j$-invariant of the given Weierstrass equation is
  invariant under the automorphism of $\Q(t)$ taking
  $t \to (12-t)/(1+4t)$, and in fact the surface is isomorphic to its base
  change by this map.  This is as expected, since we know that there is an
  Atkin-Lehner involution acting on the surface.  The fixed fibres are
  at $t = -2, 3/2$.
\end{remark}

We find that, in addition to the two triple fibres of type ${}_3I_0$,
the elliptic surface has two fibres of type $I_4$ and four of type $I_1$.
Letting the reduced subschemes of the triple fibres be $R_1,R_2$,
we use Kodaira's canonical bundle formula \cite[Theorem V.12.1]{bhpv}
to compute that the canonical bundle is $2R_1 + 2R_2 - F$, where $F$
is the class of a fibre.  Thus the dimension of the linear system
$|nK|$ is $\lfloor\frac{n}{3}\rfloor + a(n)$, where $a(n) =1,0,1$
for $n \equiv 0,1,2 \bmod 3$ respectively.

We now return to the full Atkin-Lehner quotient.
This allows us to determine $\E_6$ rigorously, even though we cannot
do the computation directly in characteristic $0$.  Namely:

\begin{thm} $\E_6$ is a surface with a genus-$1$ fibration and a
  double and a triple fibre.  In suitable coordinates,
  its Jacobian has $j$-invariant
  $\frac{64t^6-1152t^4+6912t^2-13824}{t^2-8}$ and its singular fibres
  are $I_4$ at infinity, $I_1$ at $\pm \sqrt 8$, and $I_0^*$ at $3$.
\end{thm}

\begin{proof} The Atkin-Lehner operator at $\p_7$ descends to an
  involution on $\E_3$, which exchanges the two triple fibres.
  It also exchanges the singular reduced fibres in pairs, so the quotient has
  one $I_4$ and two $I_1$ fibres.  The two fixed points on the base
  may give a double fibre or an $I_0^*$ fibre,
  depending on whether the action there is a translation or a negation
  or a negation map.  The Euler characteristic of the quotient must be a
  multiple of $12$, so exactly one of them must be an $I_0^*$ fibre, and the
  canonical bundle of $\E_6$ cannot be more positive than that of $\E_3$,
  so the other must be a double fibre.

  One calculates that the Atkin-Lehner involution on the model above has
  $4$ fixed points; in the coordinates above, they lie on the fibre
  above $-2$, so that is the one that becomes an $I_0^*$.  Since the
  fibre above $3/2$ is not fixed pointwise, the automorphism must be given by
  translation there.
  This specifies it uniquely, since there is only one rational
  $2$-torsion section.  (In terms of the model above, the generic fibre
  is isomorphic to its base change by $t \to (12-t)/(4t+1)$ by an
  automorphism taking $(x:y:z)$ to $(x/((4t+1)/7)^2:y/((4t+1)/7)^3:z)$,
  which acts as the identity on the fibre at $3/2$ and as $-1$ on the
  fibre at $-2$.  This automorphism must be composed with a translation,
  since the Atkin-Lehner involution does not fix a fibre pointwise,
  and that translation can only be by the $2$-torsion section.)

  To find the $j$-invariant, we express the $j$-invariant of $\E_3$
  in terms of the invariant function $(12-t)/(1+4t)+t$.
  The singular fibres are found by evaluating the coordinates of the
  singular fibres and the fibre on which the Atkin-Lehner automorphism
  acts by translation, as described above.
\end{proof}

\begin{remark} This is a member of the family with the given
  singular fibres that is the first entry of \cite[Table 3]{herfurtner}.
\end{remark}

\begin{remark}\label{try-direct}
  One could try to do the calculation directly by
  considering forms of degree $2$ vanishing on the worst singularities;
  this gives a map to a surface of degree $32$ in $\P^{12}$ that is singular
  along two lines.  Projecting away from these we find a surface in $\P^8$ of
  degree $20$, and projecting away from two more obvious singularities gives
  yet another surface of degree $16$ in $\P^6$, which we will denote by $\E_6$.
  For the reduction of $\E_6$ mod small primes, Magma is able to
  compute, not merely the canonical divisor, but its tensor powers, and we
  find the dimension of the space of global sections of $nK$ to be
  $0,1,1,1,1,2$ for $1 \le n \le 6$, and that the $6$-canonical map is a
  genus-$1$ fibration.  Such a surface cannot have multisection degree less
  than $6$, and indeed we can find a $6$-section.  However, it is difficult
  to do this calculation rigorously in characteristic $0$.
\end{remark}

The quotient by the product of the two Atkin-Lehner operators still has
$\hh^{2,0} = 1$ and is of general type.  It is a double cover of $\E_6$
branched along a $12$-section of self-intersection $-4$.  Unfortunately,
we are not able to exhibit a correspondence between this and a K3 surface.
We leave this problem for future work.

\subsection{Level $4\p_5$ over $\Q(\sqrt{5})$}\label{ex:d5-4p5}
This surface is similar in some ways to the surface considered in
Section~\ref{ex:d8-2p7}, since as in that example the level is the product
of the square of a prime and a prime.  However, there is an important
difference in that the eigenvalues for the weight-$2$ cusp form
are $+1$ for all Atkin-Lehner operators, so the fact that the quotient is
not of general type allows us
to prove the Oda-Hamahata conjecture in this case.
The group of Atkin-Lehner operators
has order $4$ and there are $6$ cusps.
The ring of Hilbert modular forms
is generated by $7$ forms of weight~$2$ and $3$ of weight~$4$.
The Atkin-Lehner at $4$ fixes two
of the cusps, which allows the quotient not to be of general type, in
light of Lemma~\ref{lem:fixed-point-quotient}.

In this case we will content ourselves with exhibiting a correspondence
over $\Q(\zeta_{20})^+$.  To be precise:

\begin{thm}\label{thm:d5-4p5}
  Let $S$ be the Hilbert modular surface $H_{4\p_5,\Q(\sqrt{5})}$ and
  let $K$ be the Kummer surface of $E \times E'$, where $E, E'$ are
  the elliptic curves \elllmfdb{2.2.5.1}{80.1}{a/5} and
  \elllmfdb{2.2.5.1}{80.1}{a/7}.  Then there is a dominant map
  $S \to K$ defined over $\Q(\zeta_{20})^+$, and in particular there is
  a correspondence embedding the transcendental lattice of $K$ into
  that of $S$.
\end{thm}

Since these two curves are conjugate over $\Q$, we expect that such a
correspondence exists even over $\Q$, but we have not shown this.
We note that this argument is considerably easier than some of the
previous ones even though we are working with a surface of general type.

\begin{proof}
As usual, we start by mapping to $\P^7$ by forms of degree $2$ invariant
under the involutions, obtaining a surface of degree $16$.  Projecting
away from the two cusps gives a surface of degree $10$, and although this
is not actually a canonical model Magma's {\tt CanonicalSheaf} command
produces a divisor class for which the associated map is a genus-$1$
fibration.  The reducible fibre types are $I_1^*, I_0^*, I_0^*, I_4$, and
there is a $2$-torsion point; the Weierstrass equation is
\begin{align*}
  y^2 &= x^3 + (-2t^4 + 32t^3 + 172t^2 - 1168t + 1254)x^2 + \\
  &\qquad(t^8 + 32t^7 + 276t^6 - 368t^5 - 10050t^4 + 6848t^3 + 
    126916t^2 - 295152t + 184041)x.\\
\end{align*}
Passing to the quotient gives us a fibration
with fibres $I_2^*, I_0^*, I_0^*, I_2, I_2$ and full level-$2$ structure,
and by a further $2$-isogeny we reach a surface whose reducible fibres
are $I_4^*,I_0^*,I_0^*$ and that has a $2$-torsion section.  This is already
a type of fibration that occurs on the Kummer surface of a product of two
elliptic curves, and so we can identify it using the formulas of
\cite[4.2]{kuwata-shioda}.

Over $\Q(\zeta_{20})^+$, the elliptic curves
\elllmfdb{2.2.5.1}{80.1}{a/5}, \elllmfdb{2.2.5.1}{80.1}{a/7} have
full level-$2$ structure, so we see that the elliptic curves are quadratic
twists of $y^2 = x(x-1)(x-\lambda_i)$ for
$$\lambda_1 = 20\alpha^3+32\alpha^2-32\alpha - 39,\quad
28\alpha^3 - 32\alpha^2 - 104\alpha + 121$$
by $-3\alpha^2+2\alpha+7$ and $2\alpha^3 + 3\alpha^2 - 6\alpha - 8$
respectively, where $\alpha = \zeta_{20}+\zeta_{20}^{-1}$.
Substituting into the formula for a fibration with $I_4^*$ and two $I_0^*$
fibres and changing coordinates on the base, we check that this one is
isomorphic to the image of the Hilbert modular surface as described above.
\end{proof}

\begin{remark} The argument above cannot be trivially modified to produce
  a correspondence over $\Q$, since it seems not to be possible to define
  the appropriate fibration on the Kummer surface over $\Q$.  It may be
  possible to do so over $\Q(\sqrt 5)$, however.
\end{remark}

\subsection{Level $\p_5\p_{11}$ over $\Q(\sqrt{5})$}\label{ex:d5-p5p11}
We consider the Hilbert modular surface of level $\p_5\p_{11}$ over
$\Q(\sqrt{5})$.  Since $5$ is ramified, this is well-defined up to
conjugacy.  The level is not Galois stable, so the conjecture was not
previously known.  We will show that the Hilbert modular surface admits a
dominant map to the Kummer surface of $E \times E'$ over $\Q({\sqrt 5})$,
where $E, E'$ are the elliptic curves \elllmfdb{2.2.5.1}{55.1}{a/4},
\elllmfdb{2.2.5.1}{55.2}{a/6}.

As usual, we start by using the software to compute the ring of modular
forms and the Atkin-Lehner involutions.  The standard model has $K^2 = 4$,
as noted; 
since the fixed points
of the Atkin-Lehner involutions are isolated and do not include cusps,
by Corollary \ref{cor:canonical-self-int} one might expect the quotient
to have $K^2 = 1$ and thus to be of general type.  However, we will
see that this is not the case.

The reason for this is that the elliptic points are among the fixed points
of $w_{\p_5}$.  The software computes that
there are $4$ elliptic points, two each of types $(5;1,2)$ and $(5;1,3)$.
These are not canonical singularities (a cyclic quotient singularity is
canonical if and only if its type is of the form $(n;d,n-d)$ for some
$d < n$).  They are in fact the points where all modular
forms of weight~$2$ and $4$ vanish, and one computes that they are
fixed by $w_{\p_5}$.  Thus Lemma~\ref{lem:fixed-point-quotient} does not apply.

\begin{prop}\label{prop:surprise-k3}
  The full Atkin-Lehner quotient is birational to a K3 surface $S_5$ in
  $\P^5$ that has $12$ ordinary double points and no other singularities.
  These points comprise one orbit defined over $\Q(\sqrt{5},\sqrt{11})$,
  one over $\Q(\sqrt{5},i)$, and two over $\Q(\sqrt{5})$.
\end{prop}

\begin{proof}
  Embedding the quotient by $\O(2)$, we find a surface of degree $16$ in
  $\P^7$; projecting away from the cusp, we have a surface singular
  along a line.  The projection has degree
  $7$ in $\P^4$ and is a nonminimal surface which is the blowup of a
  K3 surface at a smooth point.  Magma readily finds the minimal model
  and verifies that the singularities are as claimed.
\end{proof}

\begin{prop}\label{prop:d5d4a2a2}
  There is a genus-$1$ fibration on $S_5$ whose reducible fibres are of type
  $I_1^*, I_0^*, I_3, I_3$ and whose Mordell-Weil group is of rank $3$ and
  defined over $\Q(\sqrt{5})$.
\end{prop}

Since the expected Picard rank of
the K3 surface in this case is $18$, these are all of the generators
(and this can be checked by verifying that the zeta function of the reduction
mod $19$ has an irreducible factor of degree $4$).  The Mordell-Weil lattice
has discriminant $1$, indicating that the Picard lattice discriminant is
$144 = 16 \cdot 3^2$, and one checks that it is contained in the Picard
lattice of the Kummer surface of a product with index $3$.  The entire
Picard lattice is defined over $\Q(\sqrt 5)$, and the rational sublattice
has rank $13$.  Unfortunately
there does not seem to be a fibration defined over $\Q$ with multisection
degree $3$.

\begin{prop}\label{prop:fib-d7d4a1s}
  There is a genus-$1$ fibration defined over $\Q$ on $S_5$ with a
  section over $\Q(\sqrt{5})$, with multisection degree $2$, and with 
  reducible fibres $I_3^*, I_0$, and three $I_2$.  The rational sublattice of
  the Picard lattice of the Jacobian has rank $16$.
\end{prop}

This is done by computation. As usual, the reader is referred to
\cite{code} for the details; the equation of the general fibre is too
large to appear here.  

\begin{defn} Let $S'$ be the Jacobian of the fibration of
  Proposition \ref{prop:fib-d7d4a1s}.
\end{defn}

We again attempt to find a fibration on $S'$ defined over $\Q$ with
multisection degree $3$, but without
success.  There are rational fibrations with multisection degree $2$.
However, computing equations for such a fibration, we found that the Jacobian
did not appear to have any fibrations defined over $\Q$ with no rational
section, or any that admit an isogeny that would reduce the discriminant,
and so we were unable to proceed.  Instead we will work over $\Q(\sqrt 5)$.

\begin{prop} There is a genus-$1$ fibration on $S'$ defined over
  $\Q({\sqrt 5})$ with multisection degree
  $3$ whose Jacobian is isomorphic to the Kummer surface of
  $E \times E'$ over $\Q({\sqrt 5})$,
where $E, E'$ are the elliptic curves \elllmfdb{2.2.5.1}{55.1}{a/4},
\elllmfdb{2.2.5.1}{55.2}{a/6}.
\end{prop}

\begin{proof} Again, this is proved by computation.  We find a good model
  of $S'$ in $\P^7$ and study the curves on it and the ADE configurations
  that they form, eventually discovering a fibration with multisection
  degree $3$ and two $I_2^*$ and four $I_2$ fibres.  This is the same as
  one of the configurations on the Kummer surface, so we attempt to match
  them using the formulas for fibrations from \cite{shtukas}.  This is
  successful with one of the $9$ types of ${I_2^*}^2 I_2^4$ fibration
  for the indicated elliptic curves.
\end{proof}

The following summarizes the contents of this section.

\begin{thm}\label{conj:d5-n55}
  The Oda-Hamahata conjecture holds for elliptic curves over $\Q(\sqrt{5})$
  with discriminant $\sqrt{5}(4-\sqrt{5})$, in the sense that the desired
  correspondence exists over $\Q(\sqrt{5})$.
\end{thm}

\subsection{Level $\p_{13}$ over $\Q(\sqrt{3})$}\label{ex:d12-n13}
For our final example, we briefly discuss a case in which the field has narrow
class number $2$ rather than $1$ and the geometric genus is $2$: namely,
we choose the level to be a prime ideal of norm $13$ in $\Z[\sqrt{3}]$.
(I thank Dami\'an Gvirtz-Chen for the suggestion to study such examples.)
The Hilbert modular surface has two components and an Atkin-Lehner operator
of level $\p_{13}$ acts on their union.  Since $\p_{13}$ has a totally
positive generator $4 + \sqrt{3}$, the Atkin-Lehner operator preserves the
components rather than exchanging them, and there is no reason for
these to be isomorphic.

\begin{prop}\label{prop:two-components-reducible}
  Both components admit dominant rational maps to K3 surfaces of
  Picard rank $18$, the positive component of discriminant $-196$ and
  the negative component of discriminant $-784 = 4 \cdot -196$.
  These components have the same number of $\F_p$-points for $p < 100$
  of good reduction.
\end{prop}

\begin{proof} This is a straightforward application of techniques used
  many times already.  
After several steps we find a map from the positive component to a K3 surface
of Picard rank $18$ and discriminant $-196$ (it has a fibration with
reducible fibres of type $I_2, I_4, I_4, I_6$ and Mordell-Weil lattice of
rank $4$ and discriminant $49/3$).  On the other hand, the negative component
requires more work, as our initial constructions of elliptic fibrations have
very large rank, which would make it difficult to find a Mordell-Weil
basis.  By finding additional fibrations not defined over $\Q$ we are,
however, able to find a basis for the Picard lattice.
\end{proof}

\begin{remark} The Picard group of the negative component is defined over
  a large number field, which would make it difficult to compute directly
  with the curves on the surface.  However, by reducing modulo a suitable
  prime we do not change the Picard lattice (since the rank is even) and
  this calculation is much easier.  Since the Picard discriminant is
  a multiple of $7^2$ we do not expect to be able to complete the computation
  and do not attempt to determine the Galois action on the Picard lattice.
\end{remark}

The Hilbert modular form of level $\p_{13}$ appears to be connected with an
elliptic curve $E/\Q(\zeta_{12})$ which is a ``$\Q(\sqrt 3)$-curve'', i.e.,
it is $2$-isogenous to its Galois conjugate over $\Q(\sqrt 3)$.
The curve $E$ is defined by the equation
$$y^2 + (-\zeta_{12}^2 - \zeta_{12} + 1) xy + (\zeta_{12}^3+\zeta_{12}^2)y = x^3 + (\zeta_{12}^3+\zeta_{12}^2)x^2$$ and has conductor $\p_{13} \Z[\zeta_{12}]$.  To be precise about
the connection, fix a prime $\p$ of $\Z[\sqrt{3}]$, let $a_\p$ be the
trace of Frobenius for $E \bmod \p$, and let $e_\p$ be the Hecke eigenvalue
for the Hilbert modular form.
If $\p$ is inert in $\Q(\zeta_{12})/\Q(\sqrt{3})$, then $a_\p = e_\p^2-N(\p)$;
otherwise $a_\p = e_\p$.

The torsion subgroup of $E$ has order $14$, which is presumably related to
the factor of $7^2$ in the discriminant.  This factor makes it
difficult to relate the two K3 surfaces to each other or to the Kummer surface
of $E \times E^\sigma$, but we certainly expect such relations to exist.
In particular, we might ask whether there is a correspondence between the
K3 surfaces arising from the positive and negative quotients; as above,
there is no Atkin-Lehner operator relating them, so there is no obvious
way to construct one.

Just as $\Q$-curves over totally real fields correspond to base change
Hilbert modular forms, these calculations suggest a generalization of
Hamahata's conjecture to certain
$\Q(\alpha)$-curves:

\begin{conj} Let $f$ be a Hilbert cuspidal newform of level $I$ over the
  field $K = \Q(\sqrt{D})$ such that there exist a quadratic field
  $\Q(\sqrt{n})$ and a quadratic extension $L/K$ such that $e_\p(f)$ is
  rational for $\p$ split in $L/K$ and $\sqrt{n}$ times a rational number
  for $\p$ inert in $L/K$.  Then there exist an elliptic curve $E/L$ which
  is isogenous to its Galois conjugate over $K$, a K3 surface $S$ defined
  over $\Q$ and $K$-isogenous to the Kummer surface of $E \times E^\sigma$
  and admitting a correspondence to the Hilbert modular surface $H_{I,K}$,
  where $I$ is the conductor of $E$ viewed as an $\O_K$-ideal.
\end{conj}

\bibliography{oda}

\begin{thebibliography}{10}

\bibitem{hmf-mult}
Eran Assaf, Angelica Babei, Ben Breen, Sara Chari, Edgar Costa, Juanita
  Duque-{R}osero, Aleksander Horawa, Jean Kieffer, Avinash Kulkarni, Grant
  Molnar, Abhijit Mudigonda, Michael Musty, Sam Schiavone, Shikhin Sethi,
  Samuel Tripp, and John Voight.
\newblock Computing with {F}ourier expansions of {H}ilbert modular forms.
\newblock to appear, 2024.

\bibitem{hmf}
Eran Assaf, Angelica Babei, Ben Breen, Sara Chari, Edgar Costa, Juanita
  Duque-{R}osero, Aleksander Horawa, Jean Kieffer, Avinash Kulkarni, Grant
  Molnar, Abhijit Mudigonda, Michael Musty, Sam Schiavone, Shikhin Sethi,
  Samuel Tripp, and John Voight.
\newblock Software for computing with {H}ilbert modular forms, 2024.
\newblock \url{https://github.com/edgarcosta/hilbertmodularforms}.

\bibitem{bhpv}
Wolf~P. Barth, Klaus Hulek, Chris A.~M. Peters, and Antonius Van~de Ven.
\newblock {\em Compact complex surfaces}, volume~4 of {\em Ergeb. Math.
  Grenz\-geb., 3. Folge}.
\newblock Berlin: Springer, 2nd enlarged edition, 2004.

\bibitem{bsv}
Samuel Boissi{\`e}re, Alessandra Sarti, and Davide~Cesare Veniani.
\newblock On prime degree isogenies between {K}3 surfaces.
\newblock {\em Rend. Circ. Mat. Palermo (2)}, 66(1):3--18, 2017.

\bibitem{magma}
Wieb Bosma, John Cannon, and Catherine Playoust.
\newblock The {M}agma algebra system. {I}. {T}he user language.
\newblock {\em J. Symbolic Comput.}, 24(3-4):235--265, 1997.
\newblock Computational algebra and number theory (London, 1993).

\bibitem{BCDT}
Christophe Breuil, Brian Conrad, Fred Diamond, and Richard Taylor.
\newblock On the modularity of elliptic curves over {{\(\mathbb Q\)}}: wild
  3-adic exercises.
\newblock {\em J. Am. Math. Soc.}, 14(4):843--939, 2001.

\bibitem{proleg}
J.~W.~S. Cassels and E.~V. Flynn.
\newblock {\em Prolegomena to a middlebrow arithmetic of curves of genus 2},
  volume 230 of {\em Lond. Math. Soc. Lect. Note Ser.}
\newblock Cambridge: Cambridge Univ. Press, 1996.

\bibitem{CLO}
David Cox, John Little, and Donal O'Shea.
\newblock {\em Ideals, varieties, and algorithms. {An} introduction to
  computational algebraic geometry and commutative algebra}.
\newblock Undergraduate Texts Math. New York: Springer-Verlag, 1992.

\bibitem{cox-zucker}
David~A. Cox and Steven Zucker.
\newblock Intersection numbers of sections of elliptic surfaces.
\newblock {\em Invent. Math.}, 53:1--44, 1979.

\bibitem{FLS}
Nuno Freitas, Bao~V. Le~Hung, and Samir Siksek.
\newblock Elliptic curves over real quadratic fields are modular.
\newblock {\em Invent. Math.}, 201(1):159--206, 2015.

\bibitem{greenberg-voight}
Matthew Greenberg and John Voight.
\newblock Computing systems of {Hecke} eigenvalues associated to {Hilbert}
  modular forms.
\newblock {\em Math. Comput.}, 80(274):1071--1092, 2011.

\bibitem{Grundman1992}
H.~G. Grundman.
\newblock Defects of cusp singularities and the classification of {Hilbert}
  modular threefolds.
\newblock {\em Math. Ann.}, 292(1):1--12, 1992.

\bibitem{Grundman1994}
H.~G. Grundman.
\newblock Defect series and nonrational {H}ilbert modular threefolds.
\newblock {\em Mathematische Annalen}, 300(1):77--88, 1994.
\newblock \url{http://eudml.org/doc/165242}.

\bibitem{hamahata}
Yoshinori Hamahata.
\newblock Hilbert modular surfaces with {{\(p_g\leq 1\)}}.
\newblock {\em Math. Nachr.}, 173:193--236, 1995.

\bibitem{herfurtner}
Stephan Herfurtner.
\newblock Elliptic surfaces with four singular fibres.
\newblock {\em Math. Ann.}, 291(2):319--342, 1991.

\bibitem{keum}
JongHae Keum.
\newblock A note on elliptic {K3} surfaces.
\newblock {\em Trans. Am. Math. Soc.}, 352(5):2077--2086, 2000.

\bibitem{kw}
Chandrashekhar Khare and Jean-Pierre Wintenberger.
\newblock Serre's modularity conjecture. {I}, {II}.
\newblock {\em Invent. Math.}, 178(3):485--586, 2009.

\bibitem{kollar}
J{\'a}nos Koll{\'a}r.
\newblock {\em Lectures on resolution of singularities}, volume 166 of {\em
  Ann. Math. Stud.}
\newblock Princeton, NJ: Princeton University Press, 2007.

\bibitem{kuwata-shioda}
Masato Kuwata and Tetsuji Shioda.
\newblock Elliptic parameters and defining equations for elliptic fibrations on
  a {Kummer} surface.
\newblock In {\em Algebraic geometry in East Asia---Hanoi 2005. Proceedings of
  the 2nd international conference on algebraic geometry in East Asia, Hanoi,
  Vietnam, October 10--14, 2005}, pages 177--215. Tokyo: Mathematical Society
  of Japan, 2008.

\bibitem{shtukas}
A.~Logan and J.~Weinstein.
\newblock Higher modularity of elliptic curves over function fields.
\newblock arXiv preprint 2211.11149, 2022.

\bibitem{code}
Adam Logan.
\newblock Magma scripts supporting the claims of this paper.
\newblock \url{https://tinyurl.com/oda-hamahata}, also available from
  \url{https://sites.google.com/view/adamlogan/home}, 2024.

\bibitem{morrison-large}
D.~R. Morrison.
\newblock On {K3} surfaces with large {Picard} number.
\newblock {\em Invent. Math.}, 75:105--121, 1984.

\bibitem{morrison}
David~R. Morrison.
\newblock Isogenies between algebraic surfaces with geometric genus one.
\newblock {\em Tokyo J. Math.}, 10(1-2):179--187, 1987.

\bibitem{oda}
Takayuki Oda.
\newblock {\em Periods of {Hilbert} modular surfaces}, volume~19 of {\em Prog.
  Math.}
\newblock Birkh{\"a}user, Cham, 1982.

\bibitem{reid-godeaux}
Miles Reid.
\newblock Godeaux and {C}ampedelli surfaces.
\newblock Unpublished notes,
  \url{https://homepages.warwick.ac.uk/~masda/surf/more/Godeaux.pdf}.

\bibitem{ypg}
Miles Reid.
\newblock Young person's guide to canonical singularities.
\newblock Algebraic geometry, {Proc}. {Summer} {Res}. {Inst}.,
  {Brunswick}/{Maine} 1985, part 1, {Proc}. {Symp}. {Pure} {Math}. 46, 345-414,
  1987.

\bibitem{ribet}
Kenneth~A. Ribet.
\newblock Abelian varieties over {{\(\mathbb{Q}\)}} and modular forms.
\newblock In {\em Modular curves and Abelian varieties. Based on lectures of
  the conference, Bellaterra, Barcelona, July 15--18, 2002}, pages 241--261.
  Basel: Birkh{\"a}user, 2004.

\bibitem{roberts-schmidt}
Brooks Roberts and Ralf Schmidt.
\newblock {\em Local newforms for {{\(\text{GSp}(4)\)}}}, volume 1918 of {\em
  Lect. Notes Math.}
\newblock Berlin: Springer, 2007.

\bibitem{schutt-shioda}
Matthias Sch{\"u}tt and Tetsuji Shioda.
\newblock {\em Mordell-{Weil} lattices}, volume~70 of {\em Ergeb. Math.
  Grenz\-geb., 3. Folge}.
\newblock Singapore: Springer, 2019.

\bibitem{shimada}
I.~Shimada.
\newblock Table of all ${A}{D}{E}$-types of singular fibers of elliptic
  surfaces and the torsion parts of their {M}ordell-{W}eil groups.
\newblock
  \url{http://www.math.sci.hiroshima-u.ac.jp/shimada/preprints/EllipticK3/Table.pdf}.

\bibitem{SI}
T.~Shioda and H.~Inose.
\newblock On singular {K3} surfaces.
\newblock Complex {Analysis} and {A}lgebraic {Geometry}, A {Collection} of
  {Papers} dedicated to {K}. {Kodaira}, 119-136 (1977)., 1977.

\bibitem{shioda-mitani}
Tetsuji Shioda and Naoki Mitani.
\newblock Singular abelian surfaces and binary quadratic forms.
\newblock Classification of {A}lgebraic {Varieties} and {C}ompact {C}omplex
  {Manifolds}, {Mannheimer} {Arbeitstagung}, {Lect}. {Notes} {Math}. 412,
  259-287 (1974)., 1974.

\bibitem{silverman}
Joseph~H. Silverman.
\newblock {\em The arithmetic of elliptic curves}, volume 106 of {\em Grad.
  Texts Math.}
\newblock New York, NY: Springer, 2nd ed. edition, 2009.

\bibitem{risingsea}
Ravi Vakil.
\newblock The rising sea: foundations of algebraic geometry, 2024.
\newblock \url{http://math.stanford.edu/~vakil/216blog/FOAGsep0824public.pdf}.

\bibitem{vdgz}
G.~van~der Geer and Don Zagier.
\newblock The {Hilbert} modular group for the field {{\(\mathbb
  Q(\sqrt{13})\)}}.
\newblock {\em Invent. Math.}, 42:93--133, 1977.
\newblock \url{https://eudml.org/doc/142497}.

\bibitem{vdg}
Gerard van~der Geer.
\newblock {\em Hilbert modular surfaces}, volume~16 of {\em Ergeb. Math.
  Grenz\-geb., 3. Folge}.
\newblock Berlin etc.: Springer-Verlag, 1988.

\bibitem{vanhoften}
Pol van {H}often.
\newblock Classical motives.
\newblock Unpublished notes,
  \url{https://polvanhoften2.github.io/Pol\%20van\%20Hoften\%20-\%20Classical\%20Motives.pdf}.

\bibitem{williams}
Brandon Williams.
\newblock The rings of {Hilbert} modular forms for {{\(\mathbb{Q}(\sqrt{
  29})\)}} and {{\(\mathbb{Q}(\sqrt{ 37})\)}}.
\newblock {\em J. Algebra}, 559:679--711, 2020.
\newblock With an appendix by A. Logan.

\end{thebibliography}

\end{document}